\documentclass{amsart}
\calclayout
 \usepackage{amssymb,amsmath,amsfonts,epsfig,latexsym,tikz}
   \usepackage[alphabetic]{amsrefs}
 \usepackage{tikz-cd}
\usepackage{hyperref}
\usepackage{enumerate}
\usepackage{mathtools}
\usepackage{verbatim}
\usepackage[capitalize]{cleveref}
\usepackage[shortlabels]{enumitem}
\usepackage{subcaption}
\usepackage{thmtools}

\usetikzlibrary{positioning}
\usetikzlibrary{matrix}
\usetikzlibrary{decorations}
\usetikzlibrary{decorations.pathreplacing, decorations.pathmorphing, angles,quotes}
 
 \newtheorem{theorem}{Theorem}[section]
\newtheorem{definition}[theorem]{Definition}
\newtheorem{proposition}[theorem]{Proposition}
\newtheorem{lemma}[theorem]{Lemma}
\newtheorem{corollary}[theorem]{Corollary}

\newcommand{\Sym}{\operatorname{Sym}}
\newcommand{\F}{\mathbb{C}}
\newcommand{\C}{\mathbb{C}}
\renewcommand{\P}{\mathbb{P}}
\newcommand{\Z}{\mathbb{Z}}
\newcommand{\ord}{\operatorname{ord}}
\newcommand{\R}{\mathbb{R}}
\newcommand{\OO}{\mathcal{O}}
\newcommand{\Hom}{\operatorname{Hom}}
\newcommand{\Hilb}{\operatorname{Hilb}}

\theoremstyle{definition}
\newtheorem{remark}[theorem]{Remark}
\newtheorem{example}[theorem]{Example}

\begin{document}

\title{Cohomological aspects of power ideals}          
    
\author{Colin Crowley and Matt Larson}
\begin{abstract}
We show that the space of sections of any line bundle on the augmented wonderful variety of a hyperplane arrangement has the structure of a coalgebra. These coalgebras correspond to the hyperplane arrangement power ideals of Ardila and Postnikov, which include zonotopal algebras as a special case. By proving cohomology vanishing results on augmented wonderful varieties, we recover many results about zonotopal algebras. We also interpret the ``superspace'' zonotopal algebras of Rhoades, Tewari, and Wilson in terms of the sections of vector bundles on the augmented wonderful variety, and we use this interpretation to prove a formula that they conjectured for the Hilbert series of the superspace version of the central zonotopal algebra. 
\end{abstract}
 
\maketitle

\vspace{-20 pt}

\section{Introduction}

Let $L$ be a complex vector space of dimension $r$, and let $H_1, \dotsc, H_n$ be an essential hyperplane arrangement in $L$, i.e., the $H_i$ are the vanishing loci of linear forms on $L$ with $H_1 \cap \dotsb \cap H_n = 0$. The \emph{augmented wonderful variety} $W_L$ of this hyperplane arrangement is a certain compactification of $L$ which was introduced in \cite{BHMPW20a}. It is defined as an iterated blow-up of the projective completion $\mathbb{P}(L \oplus \mathbb{C})$ of $L$. A \emph{flat} is a subset $F$ of $\{1, \dotsc, n\}$ such that if $\cap_{i \in F} H_i$ is contained in $H_j$, then $j \in F$. The augmented wonderful variety is constructed by blowing up  the strict transforms of the strata corresponding to proper flats in the hyperplane at infinity $\mathbb{P}L$ of $\mathbb{P}(L \oplus \mathbb{C})$, in increasing order of dimension. 

The augmented wonderful variety is smooth and projective, and the complement of $L$ in $W_L$ is a simple normal crossings divisor whose components are indexed by proper flats. 
The action of the group $\mathbb{G}_m \ltimes L$ on $L$ extends to $W_L$. Here the $\mathbb{G}_m$ factor acts by scaling and the $L$ factor acts by translation.

Any line bundle on $W_L$ has a unique $L$-linearization: each divisor on $W_L$ is linearly equivalent to a unique Weil divisor $D$ which is disjoint from $L$. Indeed, let $D_F$ be the strict transform of the stratum in $\mathbb{P}L$ corresponding to a proper flat $F$, so the $D_F$ are the irreducible components of $W_L \setminus L$. Then the classes of the $D_F$ form a basis for the Picard group of $W_L$. 

We can view the space of sections $H^0(W_L, \mathcal{O}(D))$ as a subspace of $\operatorname{Sym} L^*$, the space of polynomials on $L$, with restrictions on their poles along the boundary. 
The $\mathbb{G}_m$ action on $W_L$ preserves each boundary divisor, so the subspace $H^0(W_L, \mathcal{O}(D))$ of $\operatorname{Sym} L^*$ is generated by homogeneous polynomials. 

Furthermore, the existence of an $L$ linearization implies that the subspace $H^0(W_L, \mathcal{O}(D))$ is closed under translation: for any polynomial $f(x) \in H^0(W_L, \mathcal{O}(D))$ and $v \in L$, $f(x + v)$ lies in $H^0(W_L, \mathcal{O}(D))$. This gives $H^0(W_L, \mathcal{O}(D))$ the structure of a coalgebra, as we now explain.

There is an action $\odot$ of $\operatorname{Sym} L$ on 
$\operatorname{Sym} L^*$, where $v \in L$ acts by its directional 
derivative $D_v$, and the symmetric product $v_1  \dotsb  
v_k$ acts by the composition $D_{v_1} \circ \cdots \circ D_{v_k}$. Given 
a graded subspace $S$ of 
$\operatorname{Sym} L^*$, we set
$$S^{\perp} = \{f \in \operatorname{Sym}L : f \odot g(0) = 0 \text{ for all }g \in S\}.$$
Then $S^{\perp}$ is a homogeneous ideal if and only if $S$ is closed 
under translation. Every homogeneous ideal $I$ is of the form 
$S^{\perp}$ for a unique graded subspace $S$ of $\operatorname{Sym} L^*$, 
called the \emph{Macaulay inverse system} of $I$. The inverse system of 
$I$ is identified with the dual of $\operatorname{Sym} L/I$. In 
particular, the dual of $H^0(W_L, \mathcal{O}(D))$ is identified with an 
algebra. See \cref{rmk:coalgebra-structure} for a more general setting where the 
sections of a line bundle have a coalgebra structure.

We will show that the graded rings obtained in this way are defined by 
certain ideals generated by powers of linear forms, generalizing the 
family of ideals studied by \cite{ArdilaPostnikov}.
Let $L_F$ be the subspace of $L$ corresponding to a flat $F$. Given an integer $a_F$ for each proper flat, the corresponding power ideal is the ideal 
$$(v^{a_F + 1} : v \in L_F \setminus \{0\}, \, F \text{ proper flat}) \subset \operatorname{Sym} L.$$
If $a_F < 0$ for some flat $F$, then the power ideal is defined to be the unit ideal. 

\begin{theorem}\label{thm:powerideal}
The inverse system of the power ideal $(v^{a_F + 1} : v \in L_F \setminus \{0\},\, F \text{ proper flat})$ of $\operatorname{Sym} L$ is the subspace $H^0(W_L, \mathcal{O}(\sum_F a_F D_F))$ of $\operatorname{Sym} L^*$. 
\end{theorem}

We now describe two particularly important divisors on $W_L$. 
The augmented wonderful variety is equipped with a map to the projective completion $\mathbb{P}(L \oplus \mathbb{C})$. Let $\alpha$ be the divisor class obtained by pulling back the hyperplane class along this map. By choosing linear forms defining each hyperplane, we obtain a map from $L$ to $\mathbb{C}^n$. Compactifying $\mathbb{C}^n$ by $(\mathbb{P}^1)^n$, there is a map from $L$ to $(\mathbb{P}^1)^n$ which extends to a map $W_L \to (\mathbb{P}^1)^n$. Let $\mathcal{O}(a_1, \dotsc, a_n)$ be the pullback of the corresponding line bundle from $(\mathbb{P}^1)^n$. 
\begin{center}
\begin{tikzcd}
& W_L \arrow[dr, "{\mathcal{O}(1, \dotsc, 1)}"] \arrow[dl, "\mathcal{O}(\alpha)", swap] \\
\mathbb{P}(L \oplus \mathbb{C}) & 
& (\mathbb{P}^1)^n
\end{tikzcd}
\end{center}

 For a vector $v \in L$, let $\rho(v)$ be the number of hyperplanes not containing $v$. If $v$ is contained in $L_F$ and not contained in any smaller stratum, then $\rho(v) = n - |F|$. 
 For $k \in \mathbb{Z}$, set 
$$Z_{L, k} = \frac{\operatorname{Sym} L}{(v^{\rho(v) + k + 1} : v \in L \setminus \{0\})}.$$
If there is a nonzero vector $v$ with $\rho(v) + k + 1 \le 0$, then 
$Z_{L, k} = 0$ by definition. This algebra is called the \emph{$k$th 
zonotopal algebra} of the hyperplane arrangement. This family of 
algebras, which was introduced by Holtz and Ron \cite{HoltzRon} for $k = 
0,-1,-2$, arises 
in a wide variety of contexts. See, for example, 
\cites{DM,PSS,SturmfelsXu, 
ArdilaPostnikov,BergetTutte,Lenz,RRT,GeoZonotopeI,GeoZonotopeII}. 
Let $C_{L, k}$ denote the inverse system of the $k$th zonotopal algebra, so $C_{L,k}$ is a subspace of $\operatorname{Sym} L^*$ which is identified with the dual vector space $Z_{L,k}^*$ to $Z_{L,k}$. Note that $C_{L,k} = \oplus_i C_{L,k}^i$ inherits a grading from the grading on $\operatorname{Sym} L^*$. 
We identify each $C_{L,k}$ with the space of sections of a line bundle on the augmented wonderful variety. We have $\mathcal{O}(\alpha) = \mathcal{O}(\sum_F D_F)$ and $\mathcal{O}(1, \dotsc, 1) = \mathcal{O}(\sum_F (n - |F|) D_F)$, so the following result is a consequence of Theorem~\ref{thm:powerideal}.

\begin{corollary}\label{cor:seczonotope}
As subspaces of $\operatorname{Sym}L^*$, $C_{L,k} = H^0(W_L, \mathcal{O}(1, \dotsc, 1)(k\alpha))$. 
\end{corollary}

A similar result appears in \cite{GeoZonotopeII}, but the interpretation there differs except when $k=0$. See Remark~\ref{rem:CP25}. We also show that, in many cases, the inverse systems of \emph{hierarchical zonotopal algebras}, a class of algebras introduced in \cite{HoltzRonXu} and studied in \cite{LenzHierarchical}, arise as the sections of line bundles on augmented wonderful varieties. See Proposition~\ref{prop:hierarchical}. 

Corollary~\ref{cor:seczonotope} recovers and illuminates several results about zonotopal algebras. For example, let $L/i$ be the contraction of $L$, the hyperplane arrangement obtained by restricting each hyperplane to $H_i$, and let $L \setminus i$ be the deletion, the hyperplane arrangement obtained by deleting $H_i$. 
If the linear form defining $H_i$ is nonzero, i.e., if $i$ is not a loop, then the augmented wonderful variety $W_{L/i}$ is naturally identified with a Cartier divisor on $W_L$ whose associated line bundle is $\mathcal{O}(0, \dotsc, 1, \dotsc, 0)$, where the $1$ is in the $i$th spot. If $\cap_{j \not= i} H_j = 0$, i.e., if $i$ is not a coloop, then there is an identification 
\begin{equation}\label{eq:deletionidentification}
H^0(W_L, \mathcal{O}(1, \dotsc, 1)(k\alpha)(-W_{L/i})) = H^0(W_{L \setminus i}, \mathcal{O}(1, \dotsc, 1)(k\alpha)) = C_{L \setminus i, k}.
\end{equation}
See Lemma~\ref{O-deletion}. Twisting the short exact sequence corresponding to the Cartier divisor $W_{L/i}$ by $\mathcal{O}(1, \dotsc, 1)(k\alpha)$, we obtain the following exact sequence if $i$ is not a loop or coloop:
\begin{equation}\label{eq:deletion}
0 \to C_{L \setminus i, k}^{\bullet - 1} \to C_{L, k}^{\bullet} \to C_{L/i, k}^{\bullet},
\end{equation}
where the first map is multiplication by a linear form whose vanishing locus is $H_i$ and the second map is restriction to $L/i$. 
This sequence is right exact if $H^1(W_{L \setminus i}, \mathcal{O}(1, \dotsc, 1)(k\alpha)) = 0$. 
If $k \ge -1$, then the higher cohomology of $\mathcal{O}(1, \dotsc, 1)(k\alpha)$ vanishes for any $L$, so this is a short exact sequence. See Corollary~\ref{cor:line}. This right exactness was first proved in \cite[Proposition 4.4]{ArdilaPostnikov}. When $k \le -2$, $H^1(W_L, \mathcal{O}(1, \dotsc, 1)(k\alpha))$ is usually nonzero. Nevertheless, the sequence \eqref{eq:deletion} is right exact when $k= -2$ as well \cite[Section 5.1]{HoltzRon}. We give an explanation for this phenomenon below.

This right exactness gives a recursion to compute the Hilbert series of $Z_{L, k}$ for any $k \ge -2$. Let $\operatorname{Hilb}(Z_{L, k}) = \sum \dim Z_{L,k}^i q^i$, where $Z_{L,k}^i$ is the $i$th graded piece of $Z_{L,k}$. Then, taking into account the degree shift, the exactness of \eqref{eq:deletion} when $k \ge -2$ implies that, if $i$ is not a loop or coloop, then
$$\operatorname{Hilb}(Z_{L, k}) = \operatorname{Hilb}(Z_{L/i, k}) + q \operatorname{Hilb}(Z_{L \setminus i, k}),$$
i.e., $\operatorname{Hilb}(Z_{L, k})$ satisfies a \emph{deletion-contraction recursion}.
This allows one to compute $\operatorname{Hilb}(Z_{L,k})$ from the case when the hyperplane arrangement is Boolean.
In particular, it implies that the Hilbert series of $Z_{L,k}$ is combinatorially determined, i.e., it only depends on the matroid of $L$.

The most heavily studied zonotopal algebras are $Z_{L, 0}, Z_{L,-1}$, 
and $Z_{L,-2}$, known as the external, central, and internal zonotopal 
algebras. 
The algebra $Z_{L,0}$ appeared independently in 
the work of Postnikov, Shapiro, and Shapiro \cite{PSS} and 
Wagner \cite{Wagner, wagnerAlgebrasRelatedMatroids1999}, 
and it is also known as the Postnikov-Shapiro 
algebra or the circulation algebra. The 
inverse system of the 
defining ideal of $Z_{L,-1}$ first appeared in approximation theory in 
the work of de Boor and H\"ollig \cite{deboorBsplinesParallelepipeds1982}, and the ideal itself was later 
introduced by de Boor, Dyn, and Ron \cite{deboorTwoPolynomialSpaces1991}. A 
theory containing the three cases $k=0,-1,-2$ was developed by Holtz and Ron 
\cite{HoltzRon}, where the term ``zonotopal algebra,'' as well as the 
adjectives ``external,'' ``central,'' and ``internal'' were coined. The family 
of algebras for arbitrary $k$ was studied by Ardila and Postnikov 
\cite{ArdilaPostnikov} concurrently with 
\cite{HoltzRon}. 

In the external, central, and internal cases, the Hilbert series of the 
zonotopal algebra can be described in terms of the \emph{Tutte 
polynomial} $T_L(x,y)$ of the hyperplane arrangement. The Tutte 
polynomial is a combinatorial invariant of a hyperplane arrangement 
which satisfies a deletion-contraction recursion. The Hilbert series of 
$Z_{L, 0}$, $Z_{L, -1}$, and $Z_{L, -2}$ are $q^{n-r} T_L(1 + q, 
q^{-1})$, $q^{n-r}T_L(1, q^{-1})$, and $q^{n-r}T_L(0, q^{-1})$, 
respectively \cite{ArdilaPostnikov, HoltzRon}\footnote{All three 
    formulas are claimed in {\cite[Corollary 4.13, Corollary 
    4.14, and Corollary 4.15]{ArdilaPostnikov}}. However, the 
proof is faulty in the internal case. It is pointed out in 
the correction {\cite{APCorrection}} that 
the formula still holds by the results of 
{\cite{HoltzRon}}.}.

The case of the internal zonotopal algebra, when $k=-2$, is the most 
subtle. For example, Holtz and Ron conjectured a spanning set for its 
inverse system in \cite[Conjecture 6.1]{HoltzRon}. This conjecture was 
disproved in \cite{APCorrection}. An explicit basis for the inverse system of $Z_{L, -2}$ was eventually found in \cite[Section 4.5]{GillespieThesis}. Even though $H^1(W_L, \mathcal{O}(1, \dotsc, 1)(-2\alpha))$ is usually nonzero, \eqref{eq:deletion} is exact when $k=-2$, as we now explain. 

\medskip

Let $\mathcal{L}_M$ be the line bundle $\omega_{W_L} \otimes \mathcal{O}(1, \dotsc, 1)(\alpha)$, where $\omega_{W_L}$ is the canonical bundle. There is an embedding of $W_L$ into a certain toric variety called the stellahedral toric variety, and $\mathcal{L}_M$ is the restriction of the line bundle corresponding to the independence polytope of the matroid dual of the matroid of $L$. We can identify $\mathcal{L}_M$ with a subsheaf of $\mathcal{O}(1, \dotsc, 1)(-\alpha)$, see Proposition~\ref{prop:canonical}. This gives a second interpretation of the inverse system of $Z_{L, k}$ as the sections of a line bundle, for $k \le -1$. 

\begin{proposition}\label{prop:LM}
For $k \le -1$, the inclusion of $\mathcal{L}_M((k+1)\alpha)$ into $\mathcal{O}(1, \dotsc, 1)(k\alpha)$ induces an isomorphism on global sections, so $C_{L, k} = H^0(W_L, \mathcal{L}_M((k+1)\alpha))$ as subspaces of $\operatorname{Sym} L^*$.
\end{proposition}

In Corollary~\ref{cor:LMvanish} below, we show that
$$H^i(W_L, \mathcal{L}_M(k\alpha)) = 0 \text{ for }i >0 \text{ and }k \ge -1.$$
Using this vanishing, we show that the sequence \eqref{eq:deletion} is 
exact when $k=-2$, recovering the fact that  $\Hilb(Z_{L,-2}) = q^{n-r} T_L(0, q^{-1})$. 

We can also use Proposition~\ref{prop:LM} to recover a result of Brion and Vergne: that the top graded piece  of the central zonotopal algebra $Z_{L, -1}^{n-r}$ is identified with the dual of the $r$th graded piece of the singular cohomology of $L \setminus \cup_{i} H_i$ \cite{BrionVerge}. See Remark~\ref{rem:BV}. 

\medskip

If $k \le -3$, then \eqref{eq:deletion} can fail to be exact, see Example~\ref{ex:MK4}. 
While the $\mathbb{G}_m$-equivariant Euler characteristic of 
$\mathcal{O}(1, \dotsc, 1)(k\alpha)$ is combinatorially determined 
(i.e., determined by the matroid of $L$), the Hilbert series of 
$H^0(W_L, \mathcal{O}(1, \dotsc, 1)(k\alpha)) = C_{L, k}$ is not 
combinatorial, at least for $k \le -6$, see \cite[Proposition 4]{APCorrection}. 

In \cite{HoltzRonXu}, the authors gave formulas for the dimensions of a generalization of zonotopal algebras called hierarchical zonotopal algebras. Using a result of Lenz \cite{LenzHierarchical} that identifies hierarchical zonotopal algebras with hyperplane arrangement power ideals, we can recover those formulas in many cases. See Section~\ref{ss:hierarchical}. 

\medskip
The interpretation in Corollary~\ref{cor:seczonotope} also gives new results about zonotopal algebras. 
Choose a general form in $L^*$, and let $TL$ denote the hyperplane in 
$L$ defined by the vanishing locus of that form. We have an induced hyperplane arrangement on $TL$, and $W_{TL}$ can be identified with a Cartier divisor on $W_L$ whose corresponding line bundle is $\mathcal{O}(\alpha)$. The restriction of $\mathcal{O}(1, \dotsc, 1)$ from $W_L$ to $W_{TL}$ is identified with the line bundle $\mathcal{O}(1, \dotsc, 1)$ on $W_{TL}$, and so we obtain an exact sequence
\begin{equation}
0 \to C_{L, k}^{\bullet - 1} \to C_{L, k+1}^{\bullet} \to C_{TL, k+1}^{\bullet} \to H^1(W_L, \mathcal{O}(1, \dotsc, 1)(k\alpha)),
\end{equation}
where the first map is multiplication by a linear form defining $TL$, and the second map is restriction to $TL$. 
If $k \ge -1$, then $H^1(W_L, \mathcal{O}(1, \dotsc, 1)(k\alpha)) = 0$ and we obtain a short exact sequence. This sequence, which can be used to compute the Hilbert series of $Z_{L, k}$ for $k \ge -1$ from the easy case of the external zonotopal algebra $Z_{L,0}$, does not appear to have been noticed before.

\medskip

Let $\Psi_L = \operatorname{Sym} L \otimes \wedge^{\bullet} L$ be the ring of polynomial differential forms on $L^*$. This ring, which is known as superspace, is bigraded, with a commutative (``bosonic'') and anticommutative (``fermionic'') part. There is a linear derivation  $d \colon \Psi_L \to \Psi_L$ which is determined by the properties that $d(v \otimes 1) = 1 \otimes v$ and $d(1 \otimes v) = 0$ for any $v \in L$. This map has bidegree $(-1, 1)$, i.e., it decreases the commutative degree by $1$ and increases the anticommutative degree by $1$. 

Quotients of superspace rings by \emph{differentially closed} ideals, i.e., ideals in $\Psi_L$ which are closed under $d$, have attracted significant attention in recent years, especially in the symmetric function literature, see \cites{KR,RW2,RW,Superspace,ACKMR}. In particular, given a homogeneous ideal $I = (f_1, \dotsc, f_r)$ in $\operatorname{Sym} L$, one often obtains an interesting ideal in $\Psi_L$ by considering its \emph{differential closure}, the smallest ideal in $\Psi_L$ which is closed under the action of $d$ and contains $I$. 

There is an action of $\Psi_L$ on $\Psi_{L^*}$ via differential operators. This gives a version of inverse systems of ideals in $\Psi_L$, see Section~\ref{sec:superspace}. In particular, for each ideal $I$ in $\Psi_L$, we obtain a subspace of $\Psi_{L^*}$ which is identified with the dual of $\Psi_L/I$, which we call the inverse system of $I$. We identify $\Psi_{L^*}$ with the space of polynomial differential forms on $L$. 

As $W_L$ is a simple normal crossings compactification of $L$, it has a log cotangent bundle $\Omega_L$ whose sections are differential forms on $W_L$ with logarithmic poles along the boundary of $W_L$. In particular, for any divisor $D$ which is disjoint from $L$, $H^0(W_L, \mathcal{O}(D) \otimes \wedge^j \Omega_L)$ can be identified with the polynomial $j$-forms on $L$ which satisfy certain vanishing conditions along the boundary. In this way, for any line bundle $\mathcal{L}$ on $W_L$, $H^0(W_L, \mathcal{L} \otimes \Omega_L)$ can be identified with a subspace of $\Psi_{L^*}$. 

\begin{theorem}\label{thm:diffclosed}
For any line bundle $\mathcal{L}$ on $W_L$, $\bigoplus_j H^0(W_L, \mathcal{L} \otimes \wedge^j \Omega_L)$ is the inverse system to a homogeneous differentially closed ideal in $\Psi_L$. 
\end{theorem}

For each $k$, let $\mathcal{Z}_{L, k}$ denote the quotient of $\Psi_L$ by the differential closure of the ideal of the $k$th zonotopal algebra. Then $\mathcal{Z}_{L, k}$ is equipped with a bigrading
$$\mathcal{Z}_{L, k} = \bigoplus_{i,j} \mathcal{Z}_{L, k}^{i,j},$$
where $i$ tracks the commutative grading and $j$ tracks the anticommutative grading. 
In \cite{Superspace}, Rhoades, Tewari, and Wilson studied $\mathcal{Z}_{L, 0}$. They computed its bigraded Hilbert series and gave a basis for its inverse system. For $k \le 0$, we show that $\mathcal{Z}_{L, k}$ can be described using the geometry of the augmented wonderful variety. Let $\mathcal{C}_{L, k}$ denote the inverse system of $\mathcal{Z}_{L, k}$, which is similarly bigraded. 

\begin{theorem}\label{thm:supersections}
For each $k \le 0$, $\mathcal{C}_{L,k} = \bigoplus_j H^0(W_L, \mathcal{O}(1, \dotsc, 1)(k\alpha) \otimes \wedge^j \Omega_L)$ as subspaces of $\Psi_{L^*}$. 
\end{theorem}

As $\wedge^r \Omega_L$ is a line bundle, which is isomorphic to $\mathcal{L}_M \otimes \mathcal{O}(-1, \dotsc, -1)$, the top anticommutative piece of a ring like the one in Theorem~\ref{thm:diffclosed} is identified with the inverse system of a graded ring. Concretely, the restriction of each section of $\mathcal{O}(1, \dotsc, 1)(k\alpha) \otimes \wedge^r \Omega_L$ to $L$ is equal to a polynomial on $L$ times the standard volume form on $L$, and the set of polynomials that occur in this way is the desired inverse system. 

In \cite{Superspace}, the authors observed that the Hilbert series of the top anticommutative piece of $\mathcal{Z}_{L, 0}$ is the same as the Hilbert series of the central zonotopal algebra. We generalize this to all $k \le 0$. Combining Theorem~\ref{thm:supersections} and Proposition~\ref{prop:LM}, we have that
$$\mathcal{C}_{L, k}^{\bullet, r} = H^0(W_L, \mathcal{O}(1, \dotsc, 
1)(k\alpha)\otimes \wedge^r \Omega_L) = H^0(W_L, \mathcal{L}_M(k\alpha)) 
= C_{L, k-1}$$
for any $k \le 0$,
and so $\mathcal{Z}_{L, k}$ ``interpolates'' between $Z_{L,k}$ and $Z_{L, k-1}$. In particular, $\mathcal{Z}_{L, -1}$ interpolates between the central zonotopal algebra and the internal zonotopal algebra, confirming a conjecture in \cite{Superspace}.

Similar to the case of zonotopal algebras, we can use our cohomological interpretation to establish an analogue of \eqref{eq:deletion}. The form of the exact sequence is slightly different, because the restriction of $\wedge^j \Omega_L$ to $W_{L/i}$ is isomorphic to an extension of $\wedge^j \Omega_{L/i}$ by $\wedge^{j-1} \Omega_{L/i}$. The choice of a splitting of the inclusion $L/i \hookrightarrow L$ induces an isomorphism of the restriction with $\wedge^j \Omega_{L/i} \oplus \wedge^{j-1} \Omega_{L/i}$. See Section~\ref{ssec:logcotangent}. 

\begin{theorem}\label{thm:superspaceexact}
For any $k \le 0$ and $i$ which is not a loop or coloop of $L$, there is an exact sequence
\begin{equation}\label{eq:superdeletion}
0 \to \mathcal{C}_{L \setminus i, k}^{\bullet - 1, \bullet} \to \mathcal{C}_{L, k}^{\bullet, \bullet} \to \mathcal{C}_{L/i, k}^{\bullet, \bullet} \oplus \mathcal{C}_{L/i, k}^{\bullet, \bullet - 1}.
\end{equation}
If $k \in \{-1, 0\}$, then this sequence is right exact. 
\end{theorem}

The first map is given by multiplication by a linear form whose 
vanishing locus is $H_i$, and the map $\mathcal{C}_{L, k}^{\bullet, 
\bullet} \to \mathcal{C}_{L/i, k}^{\bullet, \bullet}$ is given by 
restriction to $H_i$. The map $\mathcal{C}_{L, k}^{\bullet, \bullet} \to 
\mathcal{C}_{L/i, k}^{\bullet, \bullet - 1}$ depends on the choice of a 
splitting $L \cong L/i \oplus \mathbb{C} \cdot v$, and is induced by the map $\Psi_{L^*} \to 
\Psi_{(L/i)^*}$ given by contracting with $v$ and then restricting to 
$L/i$. 
In the case $k=0$, Theorem~\ref{thm:superspaceexact} was proved in \cite[Section 3.4]{Superspace}. Even in this case, the sequence \eqref{eq:superdeletion} depends on some choices. 
The most difficult part is to establish the exactness on the right, 
which Rhoades, Tewari, and Wilson do by proving by induction that a 
certain subset spans $\mathcal{C}_{L, 0}^{\bullet,\bullet}$. In the case when $k=-1$, this approach is complicated by the complexity of the inverse system of the internal zonotopal algebra. We show that \eqref{eq:superdeletion} can fail to be right exact when $k=-2$. See Example~\ref{ex:MK4}. 

In our approach, the right exactness follows from the vanishing of a certain $H^1$ group. The vanishing of this group turns out to be relatively straightforward to prove by induction. The base case is related to a strengthening of the degeneration of the logarithmic Hodge to de Rham spectral sequence for $W_L$, see Theorem~\ref{thm:pushforward} and Example~\ref{ex:degen}. As a consequence of Theorem~\ref{thm:superspaceexact}, standard techniques give the following corollary, proving \cite[Conjecture 6.1]{Superspace}. 

\begin{corollary}\label{cor:hilbert}
We have
$$\operatorname{Hilb}(\mathcal{Z}_{L, -1}) \coloneqq \sum_{i, j} \dim \mathcal{Z}_{L, -1}^{i,j} q^it^j = (1 + t)^r q^{n-r} T_L( \frac{1}{1 + t}, \frac{1}{q}).$$
\end{corollary}

For $k \ge 1$, in general $\mathcal{C}_{L, k}$ is not equal to $\bigoplus_j H^0(W_L, \mathcal{O}(1, \dotsc, 1)(k\alpha) \otimes \wedge^j \Omega_L)$. In \cite[Section 6.2]{Superspace}, Rhoades, Tewari, and Wilson observed that $\mathcal{Z}_{L, k}$ does not seem to behave well for $k \ge 1$, and they asked if there was a different extension of $Z_{L, k}$ to a quotient of superspace. For each $k \ge 1$, the space $\bigoplus_j H^0(W_L, \mathcal{O}(1, \dotsc, 1)(k\alpha) \otimes \wedge^j \Omega_L)$ is the inverse system of a differentially closed ideal, and it satisfies an analogue of Theorem~\ref{thm:superspaceexact}, including the exactness on the right. See Section~\ref{ssec:positivek}. 
In particular, this implies that its Hilbert series satisfies a deletion-contraction recursion, and so it is a specialization of the ``umbral Tutte polynomial'' of Ardila and Postnikov \cite[Proposition 4.11]{ArdilaPostnikov}. This implies that the Hilbert series is combinatorially determined. 
We also show that these algebras are well-behaved in a different way. 

For any homogeneous differentially closed ideal, the map $d \colon \Psi_L \to \Psi_L$ descends to a map of bidegree $(-1, 1)$ on the quotient by that ideal. As $d^2 = 0$, this defines a chain complex.  

\begin{theorem}\label{thm:exact}
For $k \ge -1$, $d$ is exact on $\bigoplus_j H^0(W_L, \mathcal{O}(1, \dotsc, 1)(k\alpha) \otimes \wedge^j \Omega_L)^*$ except in bidegree $(0,0)$. 
\end{theorem}

We also observe that $d$ is exact except in bidegree $(0,0)$ on the quotient by the differential closure of any homogeneous ideal, see Lemma~\ref{lem:diffclosedexact}. This gives an alternative proof of Theorem~\ref{thm:exact} when $k \in \{-1,0\}$, but not when $k \ge 1$, as there the relevant ideal of $\Psi_L$ is usually not the differential closure of an ordinary ideal. 

The exactness of $d$ implies that, for any $\ell > 0$, we have $\sum_{i 
+ j = \ell} (-1)^i \dim \mathcal{Z}^{i,j}_{L, -1} = 0$. Combining this 
with Corollary~\ref{cor:hilbert} gives some equalities satisfied by the 
coefficients of the Tutte polynomials of matroids which are realizable over $\mathbb{C}$. 
Such equalities were classified by Brylawski \cite{BrylawskiEquality}, 
and we obtain all of the nontrivial ones, so this exact sequence gives a 
``categorification'' of Brylawski's result. See 
Proposition~\ref{prop:brylawksiequalities}. 

For example, Corollary~\ref{cor:hilbert} implies that the dimension of $\mathcal{Z}^{n-r, r-1}_{L, -1}$ is equal to the beta invariant of the matroid of $L$, and the dimension of $\mathcal{Z}^{n-r-1, r}_{L, -1}$ is equal to the beta invariant of the matroid dual of the matroid of $L$. If $n >1$, then $d$ induces an isomorphism from $\mathcal{Z}^{n-r, r-1}_{L, -1}$ to $\mathcal{Z}^{n-r-1, r}_{L, -1}$, recovering the fact that, if the size of the ground set is at least two, the beta invariant of a matroid is equal to the beta invariant of its dual. 

\noindent \subsection*{Acknowledgements}
We thank Vasu Tewari for many helpful discussions and for showing us Example~\ref{ex:braid}. 
This work was conducted while the authors were at the Institute for 
Advanced Study, where 
the second author is supported by the Charles 
Simonyi Endowment and the Oswald Veblen Fund. 
The first author was partially supported by NSF grant DMS-2039316 and the Simons Foundation.

\section{Geometry of augmented wonderful varieties}

Let $E = \{1, \dotsc, n\}$, and let $r$ be the dimension of $L$. 
By choosing a defining equation for each hyperplane, we can identify $L$ with a subspace of $\mathbb{C}^E$, where the hyperplanes are the intersections of the coordinate hyperplanes with $L$. We allow $L$ to be contained in some coordinate hyperplanes, i.e., we do not assume that the matroid of $L$ is loopless.  We also allow some of the hyperplanes to coincide, i.e., we do not assume that the matroid of $L$ is simple. In this section, we describe some aspects of the geometry of augmented wonderful varieties. We also prove the cohomology vanishing theorems that underlie the main results. 

\subsection{Equivariant compactifications}
We begin by discussing equivariant compactifications of algebraic groups, and we explain why the augmented wonderful variety is an 
equivariant compactification of the additive group $L$.

Let us begin in the setting of a linear algebraic group $G$ over a field $k$.
An \emph{equivariant compactification} of $G$ is a 
proper variety $X$ which contains $G$ as a dense open set, together with an 
action $G \times X \to X$ that extends the group law $G \times G \to 
G$. For example, a toric variety is an equivariant 
compactification of an algebraic torus.

\begin{remark}\label{rmk:coalgebra-structure}
Let $D$ be a Cartier divisor on a normal equivariant compactification 
$X$ of a commutative group $G$, and assume that $D$ is supported on $X \setminus G$. Then $H^0(X, \mathcal{O}(D))$ is naturally a coalgebra. Indeed, the 
sections $H^0(X,\OO(D))$, viewed as a subspace of the coordinate ring $k[G]$ via restriction,
form a left $G$ subrepresentation of the regular representation $k[G]$.
Equivalently, $H^0(X,\OO(D)) \subseteq k[G]$ is a left $k[G]$-comodule 
which embeds into $k[G]$. Dually, we have a quotient of the algebra $k[G]^* \twoheadrightarrow H^0(X, \OO(D))^*$ by a left ideal. However, $G$ 
is commutative, so $k[G]^*$ is commutative. Therefore 
the quotient by any left ideal carries a natural ring structure.
\end{remark}

Often, $k[G]$ is naturally a graded ring, and $H^0(X, \mathcal{O}(D))$ is a graded subspace. If this happens, then the map $k[G]^* \to H^0(X, \mathcal{O}(D))^*$ factors through the graded dual of $k[G]$. 

In the case of toric varieties, \Cref{rmk:coalgebra-structure} is not 
interesting. Indeed, if $D$ is a Cartier divisor supported on the 
boundary of a normal toric 
variety $X$, then there is a basis of the sections of $\OO_X(D)$ whose 
elements correspond to lattice points in a certain polyhedron $P_D$ 
\cite[Page 66]{Ful93}. The ring $H^0(X, \OO_X(D))^*$ is 
simply the ring of $k$-valued functions on the lattice points in 
$P_D$.

We will be interested in 
equivariant compactifications of the additive group $\C^n$. For example, 
the projective closure $\C^n \subseteq \P(\C \oplus \C^n) \cong \P^n$ is 
an equivariant compactification, where the action is given by 
\[
    \C^n \times \P(\C \oplus \C^n) \to \P(\C \oplus \C^n),\quad (a_1, 
    \ldots, a_n) \times [x_0:x_1: \ldots: x_n] \mapsto [x_0:x_1 + 
    a_1x_0: \ldots :x_n + a_nx_0].
\]
If $G_i \subseteq X_i$ are equivariant compactifications, then the 
product $\prod_{i}G_i \subseteq \prod_i X_i$ is an equivariant 
compactification as well. If $G_1 \subseteq G_2$ is an algebraic 
subgroup, and $G_2 \subseteq X_2$ is an equivariant compactification, 
then the closure of $G_1$ in $X_2$ is an equivariant 
compactification of $G_1$. We can now see that the augmented wonderful variety is an 
equivariant compactification of $L$ from the following description.
\begin{proposition}
    The augmented wonderful variety $W_L$ is isomorphic to the closure of the 
    image of the map
    \[
        L \hookrightarrow \prod_{\emptyset \not= S \subseteq E} \C^S \hookrightarrow 
        \prod_{\emptyset \not= S \subseteq E} \P(\C \oplus \C^S),
    \]
    where the first inclusion is given by projection onto all nonzero 
    coordinate subspaces.
\end{proposition}
See, e.g., \cite[Section 3]{LLPP}. 
For the additive group $L$, the coordinate ring $\C[L]$ is 
the symmetric algebra $\Sym L^*$. If $\operatorname{Sym} L^*$ is given the usual grading,  then, over a general field, the algebra graded dual to the coalgebra structure on $\Sym L^*$ is a divided powers ring.
In characteristic 
zero, the graded dual is isomorphic to $\Sym L$. Therefore, by 
\Cref{rmk:coalgebra-structure}, any divisor $D$ supported on the 
boundary of an equivariant compactification $X$ of $L$ with the property that $H^0(X, \mathcal{O}(D))$ is a graded subspace of $\operatorname{Sym} L$ gives rise to a 
quotient algebra of $\Sym L$. The algebras associated to equivariant 
compactifications of $L$ were first explicitly studied in 
\cite{hassettGeometryEquivariantCompactifications1999}. In subsequent 
papers on equivariant compactifications of $L$, this association is 
called the Hassett--Tschinkel correspondence. See, for example, 
\cite{arzhantsevEquivariantCompletionsAffine2022}.

\subsection{Divisors on augmented wonderful varieties}

We first recall some facts about divisors on augmented wonderful 
varieties. The Chow ring of the augmented wonderful variety is the \emph{augmented Chow ring}, which is closely studied in \cite{BHMPW20a}. In particular, they describe the Picard group of augmented wonderful varieties.

The map $W_L \to (\mathbb{P}^1)^n$ gives rise to a family of divisor 
classes $y_1, \dotsc, y_n \in H^2(W_L)$ on $W_L$, the first Chern 
classes of the line bundles $\mathcal{O}(0, \dotsc, 1, \dotsc, 0)$. If 
$i$ is not a loop, i.e., if $L$ is not contained in $\mathbb{C}^{E 
\setminus i}$, then $W_{L/i}$ is the closure of $L \cap \mathbb{C}^{E 
\setminus i}$ in $W_L$. This identifies $W_{L/i}$ with a Cartier divisor 
on $W_L$, and this Cartier divisor represents $y_i$. 
If $i$ is a loop, the composition $W_L \to (\mathbb{P}^1)^n \to \mathbb{P}^1$, where the second map is the $i$th projection, is not surjective, and so the divisor class $y_i$ is $0$. 

The restriction map $H^2(W_L) \to H^2(W_{L/i})$ is computed in \cite[Proposition 2.28]{BHMPW20a}, where it is denoted $\varphi_i$. In particular, the authors of \cite{BHMPW20a} show that the following holds. 

\begin{lemma}\label{lem:restrictdiv}
Suppose $i$ is not a loop. Then
\begin{enumerate}
\item for any $j \not= i$, the restriction of $y_j \in H^2(W_L)$ to $H^2(W_{L /i})$ is $y_j$. 
\item the restriction of $y_i \in H^2(W_L)$ to $H^2(W_{L/i})$ is $0$. 
\item the restriction of $\alpha \in H^2(W_L)$ to $H^2(W_{L/i})$ is $\alpha$. 
\end{enumerate}
\end{lemma}

For a subset $S$ of $\{1, \dotsc, n\}$, let $L^S$ be the image of $L$ under the coordinate projection $\mathbb{C}^E \to \mathbb{C}^S$. Note that $L^{\{1, \dotsc, n\} \setminus i}$ is identified with $L \setminus i$.  The coordinate projection extends to a map $\pi_S \colon W_L \to W_{L^S}$, see \cite[Remark 3.2]{BHMPW20a}. Then the following lemma follows from the description of the pullback map in \cite[Section 3]{BHMPW20a}. 

\begin{lemma}\label{lem:pullback}
Let $S$ be a subset of $\{1, \dotsc, n\}$. Then
\begin{enumerate}
\item for any $i \in S$, $y_i \in H^2(W_L)$ is the pullback of $y_i \in H^2(W_{L^S})$ along $\pi_S$. 
\item if $\dim L^S = r$, then $\alpha \in H^2(W_L)$ is the pullback of $\alpha \in H^2(W_{L^S})$. 
\end{enumerate}
\end{lemma}

In particular, if $i$ is not a coloop, then $\alpha$ is pulled back from $W_{L \setminus i}$.

We now explain how to identify $\mathcal{L}_M$ with a subsheaf of $\mathcal{O}(1, \dotsc, 1)(-\alpha)$, as claimed in the introduction. For this, we will need to discuss the relationship between the boundary divisors $D_F$ and the $y_i$. Using the description of the class group of $W_L$ in \cite[Section 1.2]{BHMPW20a}, we have that
$$y_i = \sum_{F \not \ni i} [D_F] \text{ as divisor classes on }W_L.$$
As $\alpha$ is the pullback of the class of a hyperplane from $\mathbb{P}(L \oplus \mathbb{C})$, and the reduced union of the $D_F$ is the preimage of the hyperplane at infinity, we have
$$\alpha = \sum_{F\not= E} [D_F] \text{ as divisor classes on }W_L.$$
See also \cite[Section 1.4]{BHMPW20a}. Combining these, we see that 
\begin{equation}\label{eq:centraldiv}
c_1(\mathcal{O}(1, \dotsc, 1)(-\alpha)) = y_1 + \dotsb + y_n - \alpha = \sum_{F \not= E} (n - |F| - 1) [D_F].
\end{equation}
For a flat $F$, let $\operatorname{rk}(F)$ be the codimension of the corresponding subspace $L_F$. 
\begin{proposition}\label{prop:canonical}
We have 
$$\omega_{W_L} = \mathcal{O}(\sum_{F\not= E} (\operatorname{rk}(F) 
    - r - 1) D_F).$$
In particular, we have
$$c_1(\mathcal{L}_M) = c_1(\omega_{W_L}) + y_1 + \dotsb + y_n + \alpha = \sum_{F \not= E} (n - |F| - r + \operatorname{rk}(F))[D_F].$$
\end{proposition}
Combining the second part of Proposition~\ref{prop:canonical} with \eqref{eq:centraldiv}, we see that
$$c_1(\mathcal{O}(1, \dotsc,1)(-\alpha)) - c_1(\mathcal{L}_M) = \sum_{F \not= E} (r - \operatorname{rk}(F) - 1)[D_F],$$
which is effective. This means that $\mathcal{L}_M$ can be identified with a subsheaf of $\mathcal{O}(1, \dotsc, 1)(-\alpha)$. This justifies the claim in the introduction that the sections of $\mathcal{L}_M((k+1)\alpha)$ are a subset of the sections of $\mathcal{O}(1, \dotsc, 1)(k\alpha)$. 

\begin{proof}[Proof of Proposition~\ref{prop:canonical}]
We use the description of $W_L$ as an iterated blow-up. The pullback of the canonical bundle from $\mathbb{P}(L \oplus \mathbb{C})$ is $\mathcal{O}(-(r + 1)\alpha)$. At each step of the construction of $W_L$, we blow up the strict transform of $\mathbb{P}L_F$, which has codimension $\operatorname{rk}(F) + 1$. The result then follows from \cite[Exercise 22.4.S]{VakilNotes}. 
\end{proof}

\begin{lemma}\label{lem:restrictLM}
Suppose $i$ is not a loop, and let $\mathcal{L}_{M/i}$ be the line bundle  $\omega_{W_{L/i}}(1,\dotsc, 1)(\alpha)$ on $W_{L/i}$. 
The restriction of $\mathcal{L}_M$ to $W_{L/i}$ is isomorphic to $\mathcal{L}_{M/i}$. 
\end{lemma}

\begin{proof}
It follows from Lemma~\ref{lem:restrictdiv}(2) and adjunction that the 
restriction of $\omega_{W_L}$ to $W_{L/i}$ is isomorphic to 
$\omega_{W_{L/i}}$. The result then follows from 
Lemma~\ref{lem:restrictdiv}(1) and (3). Alternatively, it can be easily 
checked using Proposition~\ref{prop:canonical} and \cite[Proposition 
2.28]{BHMPW20a}. 
\end{proof}

\medskip

For applications to hierarchical zonotopal algebras, we will need to consider a family of divisors associated to \emph{order filters} in the lattice of flats of the hyperplane arrangement, i.e., sets $\mathcal{J}$ such that if $F \in \mathcal{J}$ and $G \supset F$, then $G \in \mathcal{J}$. Given an order filter $\mathcal{J}$, consider the divisor class
\begin{equation}\label{eq:divisororder}
D_\mathcal{J} = \sum_{F \in \mathcal{J}, \, F \not= E} [D_F].  
\end{equation}
Note that if $\mathcal{J} = \{E\}$, then $D_{\mathcal{J}} = 0$, and if $\mathcal{J}$ contains all flats, then $D_{\mathcal{J}} = \alpha$. 
In Proposition~\ref{prop:hierarchicalvanishing} below, we prove vanishing results for line bundles related to this divisor. Our main tool will be removing a minimal element from the order filter $\mathcal{J}$; this yields a smaller order filter. 

If $L$ is loopless, then the \emph{wonderful variety} $\underline{W}_L$ is the variety obtained by blowing up $\mathbb{P}L$ along the strict transforms of strata corresponding to proper nonempty flats, in increasing order of dimension. The wonderful variety is equipped with a divisor $\underline{D}_F$ for each proper nonempty flat $F$, corresponding to the strict transform of $\mathbb{P}L_F$. There is also a divisor $\underline{\alpha}$, corresponding to the pullback of the hyperplane class on $\mathbb{P}L$.

The divisor $D_F$ on $W_L$ can be identified with $W_{L^F} \times \underline{W}_{L_F}$, see \cite[Proposition 2.7]{BHMPW20a}. In particular, each line bundle on $W_{L}$ restricts to a line bundle on $W_{L^F} \times \underline{W}_{L_F}$. We describe this restriction map. 

\begin{lemma}\label{lem:restrictToBoundary}\cite[Proposition 2.20]{BHMPW20a}
Let $F$ be a proper flat. Then
\begin{enumerate}
\item for any $i \in F$, the restriction of $y_i \in H^2(W_L)$ to $H^2(W_{L^F} \times \underline{W}_{L_F})$ is $y_i \otimes 1$. 
\item for any $i \not \in F$, the restriction of $y_i \in H^2(W_L)$ to $H^2(W_{L^F} \times \underline{W}_{L_F})$ is $0$. 
\item the restriction of $\alpha \in H^2(W_L)$ to $H^2(W_{L^F} \times \underline{W}_{L_F})$ is $1 \otimes \underline{\alpha}$.
\item if $G$ is incomparable with $F$, then the restriction of $[D_G] \in H^2(W_L)$ to $H^2(W_{L^F} \times \underline{W}_{L_F})$ is $0$. 
\item the restriction of $\sum_{G \ge F} [D_G] \in H^2(W_L)$ to $H^2(W_{L^F} \times \underline{W}_{L_F})$ is $-\alpha \otimes 1 + 1 \otimes \underline{\alpha}$. 
\end{enumerate}
\end{lemma}

\subsection{The log cotangent bundle}\label{ssec:logcotangent}

Recall that $\Omega_L$ is the log cotangent bundle of $W_L$, viewed as a compactification of $L$. A section of $\wedge^j \Omega_L$ is a $j$-form $\omega$ on $L$ such that both $\omega$ and $d\omega$ have a pole of order at most $1$ along each divisor $D_F$. 

The embedding of $L$ into $\mathbb{C}^n$ induces an embedding of $W_L$ into the augmented wonderful variety of the Boolean arrangement, i.e., the stellahedral toric variety $X_n$. 
In \cite{EHL}, the authors construct a toric vector bundle 
$\mathcal{S}_L$, called the \emph{augmented tautological subbundle}, on 
the stellahedral toric variety associated to a hyperplane arrangement $L$. By \cite[Theorem 9.2]{EHL}, the restriction of $\mathcal{S}_L^*$ to $W_L$ is isomorphic to $\Omega_L$. However, the $\mathbb{G}_m$-linearization on $\mathcal{S}_L^*|_{W_L}$ is different from the linearization on $\Omega_L$. 

The description of $\mathcal{S}_L^{*}$ in \cite{EHL} implies that $\Omega_L$ can be realized as a quotient of $\oplus \pi_i^* \mathcal{O}(-1)$, where $\pi_i$ is the map from $W_L$ to $\mathbb{P}^1$ given by composing the map $W_L \to (\mathbb{P}^1)^n$ with the $i$th projection. If the arrangement is Boolean, then $\Omega_L = \oplus \pi_i^* \mathcal{O}(-1)$. 

\begin{lemma}\label{lem:restriction}
If $i$ is not a loop, then there is a short exact sequence of sheaves on $W_{L/i}$
$$0 \to \mathcal{O}_{W_{L/i}} \to \Omega_L|_{W_{L/i}} \to \Omega_{L/i} \to 0.$$
\end{lemma}

\begin{proof}
Because the restriction of the boundary of $W_L$ to $W_{L/i}$ is the boundary of $W_{L/i}$, the logarithmic version of the conormal exact sequence is
$$0 \to N^{*} \to \Omega_{L}|_{W_{L/i}} \to \Omega_{L/i} \to 0,$$
where $N^{*}$ is the conormal bundle of $W_{L/i}$ in $W_L$, see 
\cite[1.1]{Olsson}. Because $[W_{L/i}] = y_i$ restricts to $0$ on $W_{L/i}$ by Lemma~\ref{lem:restrictdiv}(2), $N^{*}$ is isomorphic to $\mathcal{O}_{W_{L/i}}$. 
\end{proof}

Suppose that $i$ is not a loop. We will now show that the choice of a splitting of the inclusion $L/i \hookrightarrow L$ induces a splitting of the short exact sequence in Lemma~\ref{lem:restriction}. The fiber of $\Omega_L$ at the origin is identified with $L^*$, and the fiber of $\Omega_{L/i}$ at the origin is identified with $(L/i)^*$, so any splitting of the above sequence induces a splitting of $L/i \hookrightarrow L$. Because $\Omega_{L}|_{W_{L/i}}$ and $\Omega_{L/i}$ are both $L/i$-equivariant sheaves, a morphism $\Omega_{L/i} \to \Omega_L|_{W_{L/i}}$ is determined by the map on the fibers over the origin. 

\begin{lemma}\label{lem:split}
Given a splitting of the inclusion $L/i \hookrightarrow L$, there is a splitting of the sequence in Lemma~\ref{lem:restriction} such that the map on the fibers over the origin is the dual to the chosen splitting. 
\end{lemma}

To prove Lemma~\ref{lem:split}, we will construct the desired splitting 
on the stellahedral toric variety which contains $W_{L/i}$. There, we 
have a description, due to Klyachko, of the category of toric vector 
bundles on a toric variety in terms of certain $\mathbb{Z}$-indexed 
decreasing filtrations $\{F^{\rho}(k)\}_{k \in \mathbb{Z}}$ on the fiber 
over the identity, one for each ray $\rho$ of the fan \cite{KlyachkoVB}. 

\begin{proof}[Proof of Lemma~\ref{lem:split}]
It will be more convenient to split the dual of the sequence in Lemma~\ref{lem:restriction} by producing a map $\Omega_{L}^*|_{W_{L/i}} \to \Omega_{L/i}^*$. 
The rays of $X_{n-1}$, the stellahedral toric variety which contains $W_{L/i}$, are labeled by either an element $j \in E \setminus \{i\}$ or a subset $S$ contained in $E \setminus \{i\}$. The vector bundle $\mathcal{S}_{L/i}$, whose restriction to $W_{L/i}$ is isomorphic to $\Omega_{L/i}^*$, is a subbundle of $\oplus_{j \not= i} \pi_j^* \mathcal{O}(1)$. We first describe the equivariant vector bundle $\oplus_{j \not= i} \pi_j^* \mathcal{O}(1)$ (with the linearization described in \cite[Section 3.4]{EHL}) in the language of Klyachko. The fiber over the identity is identified with $\mathbb{C}^{n-1}$. The filtrations are given by
$$F^j(k) = \begin{cases} \mathbb{C}^{n-1} & k \le 0 \\ 0 & k > 0, \end{cases} \text{ and } F^S(k) = \begin{cases} \mathbb{C}^{n-1} & k \le 0 \\ \mathbb{C}^{E \setminus (S \cup i)} & k = 1 \\ 0 & k > 1. \end{cases}$$
The fiber of $\mathcal{S}_{L/i}$ over the identity of the torus in $X_{n-1}$ is the subspace $L/i$ of $\mathbb{C}^{n-1}$, so the Klyachko description of $\mathcal{S}_{L/i}$ consists of the following filtrations on $L/i$:
$$F^j(k) = \begin{cases} L/i & k \le 0 \\ 0 & k > 0, \end{cases} \text{ and } F^S(k) = \begin{cases} L/i & k \le 0 \\ L_{S \cup i} & k = 1 \\ 0 & k > 1. \end{cases}$$
Similarly, $\mathcal{S}_{L}|_{X_{n-1}}$ is a vector bundle which restricts to $\Omega_{L}^*|_{W_{L/i}}$. It is a subbundle of $\mathcal{O} \oplus \bigoplus_{j \not= i} \pi_j^* \mathcal{O}(1)$, and its Klyachko description consists of the following filtrations on $L$:
$$F^j(k) = \begin{cases} L & k \le 0 \\ 0 & k > 0, \end{cases} \text{ and } F^S(k) = \begin{cases} L & k \le 0 \\ L_{S \cup i} & k = 1 \\ 0 & k > 1. \end{cases}$$
A splitting of the inclusion $L/i \hookrightarrow L$ respects the filtrations associated to $\mathcal{S}_L|_{X_{n-1}}$ and $\mathcal{S}_{L/i}$, and so it induces a map $\mathcal{S}_{L}|_{X_{n-1}} \to \mathcal{S}_{L/i}$. The restriction of this map to $W_{L/i}$ splits the dual of the sequence in Lemma~\ref{lem:restriction}. 
\end{proof}

For future use, we note the following identity. Recall that $r = \dim L$. 

\begin{proposition}\label{prop:SD}
For any divisor $D$ on $W_L$, we have an isomorphism
$$H^q(W_L, \mathcal{O}(D) \otimes \wedge^j \Omega_L) \cong H^{r - q}(W_L, \mathcal{O}(-\alpha - D) \otimes \wedge^{r-j} \Omega_L)^*.$$
\end{proposition}

\begin{proof}
Recall that, because $\Omega_{L}$ is a vector bundle of rank $r$, there is an isomorphism
$$\wedge^j \Omega_L  \cong \wedge^{r-j} \Omega_{L}^* \otimes \det \Omega_{L}.$$
The determinant of $\Omega_L$ is the log-canonical bundle, i.e., the canonical bundle twisted by the sum of the boundary divisors. As the divisor class $\alpha$ is linearly equivalent to the sum of the boundary divisors, we have $\det \Omega_L \cong \omega_{W_L}(\alpha)$. Using Serre duality, we then have isomorphisms
\begin{equation*}\begin{split}
H^q(W_L, \mathcal{O}(D) \otimes \wedge^j \Omega_L) & \cong H^q(W_L, \mathcal{O}(\alpha + D) \otimes \wedge^{r - j} \Omega_L^* \otimes \omega_{W_L} ) \\ 
& \cong H^{r - q}(W_L, \mathcal{O}(-\alpha - D) \otimes \wedge^{r-j} \Omega_L)^*. \qedhere
\end{split}\end{equation*}
\end{proof}

\subsection{Vanishing theorems}\label{sec:vanishing}

  We now prove our main vanishing theorem for the cohomology of twists 
  of log cotangent bundles on augmented wonderful varieties. This will 
  allow us to show that various sequences of inverse systems of 
  zonotopal algebras are exact on the right.

\begin{theorem}\label{thm:vanishing}
For any $k \ge -1$, any $j \ge 0$, and any $i > 0$, we have 
$$H^i(W_L, \mathcal{O}(1, \dotsc, 1)(k\alpha) \otimes \wedge^j \Omega_L) = 0.$$
\end{theorem}

In particular, taking $j = 0$, we have the following corollary. 

\begin{corollary}\label{cor:line}
For any $k \ge -1$ and $i > 0$, we have 
$$H^i(W_L, \mathcal{O}(1, \dotsc, 1)(k\alpha)) = 0.$$
\end{corollary}

Taking $j = r$ and using that $\wedge^r \Omega_L \cong \mathcal{L}_M(-1, \dotsc, -1)$, we have the following corollary. 

\begin{corollary}\label{cor:LMvanish}
For any $k \ge -1$ and $i > 0$, we have 
$$H^i(W_L, \mathcal{L}_M(k\alpha)) = 0.$$
\end{corollary}

\begin{remark}
Recently, Berget and Fink have used Kempf collapsing of vector bundles 
\cite{Kempf, Weyman} to prove strong cohomology vanishing results for 
vector bundles related to hyperplane arrangements 
\cite{BergetFink,BF24}. See \cite{LiuVanishing} for an extension of 
these techniques and other results. At least in the case when $k=-1$, 
this technique does not appear to be sufficient to prove 
Theorem~\ref{thm:vanishing}. For example, one difficulty arises because $\Omega_L$ is usually not a subbundle of a trivial bundle. Additionally, the proof of Theorem~\ref{thm:vanishing} works over a field of arbitrary characteristic, while the techniques of Berget and Fink only work over a field of characteristic $0$. In \cite{EFL}, Eur, Fink, and the second author have proven strong vanishing theorems for certain nef line bundles on augmented wonderful varieties. These results give an alternative proof of Corollary~\ref{cor:line} and Corollary~\ref{cor:LMvanish} when $k \ge 0$, but they do not cover the case $k=-1$ or prove any other cases of Theorem~\ref{thm:vanishing}. 
\end{remark}

The main technical result in the proof of Theorem~\ref{thm:vanishing} is the following theorem. 

\begin{theorem}\label{thm:pushforward}
For any subset $S$ of $\{1, \dotsc, n\}$ and any $j \ge 0$, we have $R \pi_{S*} \wedge^j \Omega_{L} = \wedge^j \Omega_{L^S}$. 
\end{theorem}

\begin{example}\label{ex:degen}
When $S = \emptyset$, $W_{L^{\emptyset}}$ is a point, and applying $R \pi_{\emptyset *}$ to $\wedge^j \Omega_L$ is the same thing as taking the cohomology of $\wedge^j \Omega_L$. In this case, Theorem~\ref{thm:pushforward} asserts that $H^i(W_L, \wedge^j \Omega_L) = 0$ unless $i = j = 0$. This is a consequence of the degeneration of the logarithmic Hodge to de Rham spectral sequence, as the de Rham cohomology of $L$ is trivial except in degree $0$. However, the proof of Theorem~\ref{thm:pushforward} works in any characteristic, proving the degeneration of the logarithmic Hodge to de Rham spectral sequence for augmented wonderful varieties. Note that augmented wonderful varieties of hyperplane arrangements in characteristic $p$ often do not lift to characteristic $0$ or even over $W_2$. In \cite{LiuVanishing}, the author proves the degeneration of the logarithmic Hodge to de Rham spectral sequence for (non-augmented) wonderful varieties. 
\end{example}

\begin{proof}[Proof of Theorem~\ref{thm:pushforward}]
It suffices to prove the statement when $S = \{1, \dotsc, n\} \setminus 
i$. We first consider the case when $i$ is a coloop. Then the map $\pi_S 
\colon W_L \to W_{L \setminus i}$ factors through $W_{L \setminus i} 
\times \mathbb{P}^1$, and the map $W_L \to W_{L \setminus i} \times 
\mathbb{P}^1$ is an iterated blow-up at smooth centers \cite[Theorem 
1.3(2)]{LiLi}. The log cotangent bundle $\Omega_L$ is the pullback from 
$W_{L \setminus i} \times \mathbb{P}^1$ of $\Omega_{L \setminus i} 
\boxplus \mathcal{O}(-1)$. Therefore, for any $j \ge 0$, we have
$$\wedge^j \Omega_L \cong \pi_S^* \wedge^j \Omega_{L \setminus i} \oplus ( \pi_S^*\wedge^{j-1} \Omega_{L \setminus i} \otimes \pi_i^* \mathcal{O}(-1) ).$$
By the projection formula, the derived pushforward of $\wedge^j \Omega_L$ to $W_{L \setminus i} \times \mathbb{P}^1$ is $ \wedge^j \Omega_{L \setminus i} \oplus (\wedge^{j-1} \Omega_{L \setminus i} \boxtimes  \mathcal{O}(-1) )$. 
The projection formula implies that $R\pi_{S*}\pi_S^* \wedge^j \Omega_{L \setminus i} = \wedge^j \Omega_{L \setminus i}$. On the other hand,  $R\pi_{S*} \pi_S^*(\wedge^{j-1} \Omega_{L \setminus i} \boxtimes  \mathcal{O}(-1)) = 0$. Indeed, the restriction of $\pi_S^*\wedge^{j-1} \Omega_{L \setminus i} \otimes  \pi_i^*\mathcal{O}(-1)$ to any fiber of $\pi_S$ has no cohomology, so this follows from the theorem on cohomology and base change \cite[Theorem 25.1.6]{VakilNotes}. 

In the case when $i$ is not a coloop, we claim that $\pi_S$ is an isomorphism over $L$, in the sense that $\pi_S^{-1}(L \setminus i) = L$. Indeed, this is equivalent to checking that the pullback of the divisor $W_{L \setminus i} \setminus (L \setminus i)$ is the divisor $W_L \setminus L$, which holds by Lemma~\ref{lem:pullback}(2). 
The result then follows from \cite[Lemma 1.3 and Lemma 1.5]{EV} (see also \cite[Theorem 31.1]{HodgeIdeals}). 
\end{proof}

\begin{remark}
The proof of Theorem~\ref{thm:pushforward} works for augmented wonderful varieties over any field. Indeed, the only place where the assumption that the characteristic is $0$ is invoked is in the case when $i$ is not a coloop. In that case, the map $W_L \to W_{L \setminus i}$ is an iterated blow-up at strata of codimension $2$ which have simple normal crossings with the boundary, and so the hypothesis of \cite[Proposition 31.3]{HodgeIdeals} applies. Note that the proof of \cite[Proposition 31.3]{HodgeIdeals} works over a field of any characteristic. 
\end{remark}

We now use Theorem~\ref{thm:pushforward} to prove Theorem~\ref{thm:vanishing}. We begin by proving a special case of Theorem~\ref{thm:vanishing}. 

\begin{lemma}\label{lem:boolean}
Suppose that $L = \mathbb{C}^n$, i.e., the arrangement is Boolean. Then, for any $k \ge -1$, any $j \ge 0$, and any $i > 0$, we have 
$$H^i(W_L, \mathcal{O}(1, \dotsc, 1)(k\alpha) \otimes \wedge^j \Omega_L) = 0.$$
\end{lemma}

\begin{proof}
First suppose that $k \ge 0$. Note that $\wedge^j \Omega_L \cong \bigoplus_{|S| = j} \pi_S^* \mathcal{O}(-1, \dotsc, -1)$. In particular, $\mathcal{O}(1, \dotsc, 1)(k\alpha) \otimes \wedge^j \Omega_L$ is a direct sum of nef line bundles on the toric variety $W_L$. The result then follows from the fact that the higher cohomology of a nef line bundle on a toric variety vanishes \cite[pg. 74]{Ful93}.

Now suppose that $k = -1$. By Proposition~\ref{prop:SD}, we have
$$H^i(W_L, \mathcal{O}(1, \dotsc, 1)(-\alpha) \otimes \wedge^j \Omega_L) \cong H^{n-i}(W_L, \mathcal{O}(-1, \dotsc, -1) \otimes \wedge^{n-j} \Omega_L)^*.$$
The vector bundle $\mathcal{O}(-1, \dotsc, -1) \otimes \wedge^{n-j} 
\Omega_L$ is pulled back along the toric birational map from $W_L$ to 
$(\mathbb{P}^1)^n$. By the projection formula and \cite[pg. 76]{Ful93}, 
it suffices to show that
$$\bigoplus_{|S| = n-j}H^{n-i}((\mathbb{P}^1)^n, \mathcal{O}(-1, \dotsc, -1) \otimes \mathcal{O}(-\sum_{t \in S} y_t)) = 0$$
for $i > 0$. It follows from the K\"{u}nneth formula and the fact that $\mathcal{O}(-1)$ on $\mathbb{P}^1$ has no cohomology that this group vanishes except when $i = j = 0$.  
\end{proof}

\begin{proof}[Proof of Theorem~\ref{thm:vanishing}]
We say that $S \subset \{1, \dotsc, n\}$ is \emph{spanning} if $\dim L^S = r$; by Lemma~\ref{lem:pullback}(2), this implies that $\pi_S^* \alpha = \alpha$. We show by induction that if $S$ is spanning, then for any $k \ge -1, j \ge 0$, and $i > 0$, we have
\begin{equation}\label{eq:vanish}
H^i(W_L, \mathcal{O}(k \alpha + \sum_{t \in S} y_t) \otimes \wedge^j \Omega_L) = 0.  
\end{equation}
First consider the case when $S$ is a basis, i.e., $|S| = r$. By Lemma~\ref{lem:pullback}, $\mathcal{O}(k \alpha + \sum_{t \in S} y_t) = \pi_S^* \mathcal{O}(1, \dotsc, 1)(k \alpha)$. By the projection formula and Theorem~\ref{thm:pushforward}, we have
$$R \pi_{S*} (\mathcal{O}(k \alpha + \sum_{t \in S} y_t) \otimes \wedge^j 
\Omega_L) = \mathcal{O}(1, \dotsc, 1)(k \alpha) \otimes R\pi_{S*} 
\wedge^j \Omega_L = \mathcal{O}(1, \dotsc, 1)(k\alpha) \otimes \wedge^j 
\Omega_{L^S}.$$
Because $S$ is a basis, $L^S$ is Boolean, so the result follows from the Leray spectral sequence and Lemma~\ref{lem:boolean}. 

We now do the inductive step. Let $S$ be a spanning set, and choose some $s \not \in S$. 
If $s$ is a loop, then $\mathcal{O}(y_s)$ is trivial, and so the result is immediate. 
If $s$ is not a loop, then we have a short exact sequence
$$0 \to \mathcal{O}_{W_L} \to \mathcal{O}_{W_L}(y_s) \to \mathcal{O}_{W_{L/s}} \to 0.$$
Tensoring this with $\mathcal{O}(k \alpha + \sum_{t \in S} y_t) \otimes \wedge^j \Omega_L$ and using Lemma~\ref{lem:restrictdiv}, we obtain a short exact sequence
\begin{equation}\label{eq:SESinduct}
0 \to \mathcal{O}(k \alpha + \sum_{t \in S} y_t) \otimes 
\wedge^j \Omega_L \to  \mathcal{O}(k \alpha + \sum_{t \in S \cup s} 
    y_t) \otimes \wedge^j \Omega_L \to \mathcal{O}(k \alpha + 
        \sum_{t \in S} y_t) \otimes \wedge^j \Omega_{L}|_{W_{L/s}} \to 0.
\end{equation}
Taking the $j$th exterior power of the short exact sequence in Lemma~\ref{lem:restriction}, we have a short exact sequence
$$0 \to \wedge^{j-1} \Omega_{L/s} \to \wedge^j \Omega_{L}|_{W_{L/s}} \to \wedge^j \Omega_{L/s} \to 0.$$
After we tensor this with $\mathcal{O}(k\alpha + \sum_{t \in S} y_t)$ for some $k \ge -1$, the cohomology groups of the left and right terms are concentrated in degree $0$ by induction, so the same is true for $\mathcal{O}(k \alpha + \sum_{t \in S} y_t) \otimes \wedge^j \Omega_L|_{W_{L/s}}$. 
This implies that the cohomology groups of the left and right terms in \eqref{eq:SESinduct} are concentrated in degree $0$ by induction on $n$, so the same is true for the middle term, completing the inductive step. 
\end{proof}

Finally, we prove a vanishing theorem for line bundles on augmented wonderful varieties which will be useful in our discussion of hierarchical zonotopal algebras. Recall that in \eqref{eq:divisororder}, we have defined a divisor class $D_{\mathcal{J}}$ on $W_L$ for each order filter $\mathcal{J}$ on the lattice of flats of our hyperplane arrangement. 

\begin{proposition}\label{prop:hierarchicalvanishing}
Let $\mathcal{J}$ be an order filter on the lattice of flats. Then for any $k \ge -1$ and any $i > 0$, we have $H^i(W_L, \mathcal{O}(1, \dotsc, 1)(k \alpha)(D_{\mathcal{J}})) = 0$. 
\end{proposition}

\begin{proof}
We induct on the size of $\mathcal{J}$; the case when $\mathcal{J} = 
\{E\}$ is Corollary~\ref{cor:line}. Suppose that the desired vanishing 
holds for $\mathcal{J}$, and consider a flat $F$ such that $\mathcal{J}' 
= \mathcal{J} \cup \{F\}$ is an order filter. Twisting the short exact 
sequence corresponding to the Cartier divisor $D_F$ by $\mathcal{O}(1, 
\dotsc, 1)(k \alpha)(D_{\mathcal{J}'})$, we have an exact sequence
$$0 \to \mathcal{O}(1, \dotsc, 1)(k \alpha)(D_{\mathcal{J}}) \to \mathcal{O}(1, \dotsc, 1)(k \alpha)(D_{\mathcal{J}'}) \to \mathcal{O}_{D_F}(1, \dotsc, 1)(k \alpha)(D_{\mathcal{J}'}) \to 0.$$
By Lemma~\ref{lem:restrictToBoundary}, using the identification of $D_F$ with $W_{L^F} \times \underline{W}_{L_F}$, the divisor class associated to the line bundle $\mathcal{O}_{D_F}(1, \dotsc, 1)(k \alpha)(D_{\mathcal{J}'})$ on $D_F$ is $\sum_{i \in F} y_i \otimes 1 - \alpha \otimes 1 + 1 \otimes (k+1) \underline{\alpha}$. In particular, for each $i \ge 0$
\begin{equation}\label{eq:restrictioncohomology}
H^i(D_F, \mathcal{O}_{D_F}(1, \dotsc, 1)(k \alpha)(D_{\mathcal{J}'})) = \bigoplus_{a + b = i} H^{a}(W_{L^F}, \mathcal{O}(1, \dotsc, 1)(-\alpha)) \otimes H^b(\underline{W}_{L_F}, \mathcal{O}((k+1)\alpha)).
\end{equation}
If $i > 0$ and $k \ge -1$, then the right-hand side of \eqref{eq:restrictioncohomology} vanishes: by Corollary~\ref{cor:line}, $H^{a}(W_{L^F}, \mathcal{O}(1, \dotsc, 1)(-\alpha)) = 0$ unless $a = 0$, and $H^b(\underline{W}_{L_F}, \mathcal{O}((k+1)\alpha)) = H^b(\mathbb{P}L_F, \mathcal{O}(k+1))$, which vanishes unless $b = 0$. Therefore the induction hypothesis implies that the higher cohomology of $\mathcal{O}(1, \dotsc, 1)(k \alpha)(D_{\mathcal{J}'})$ vanishes. 
\end{proof}

If $k = -2$, then the higher cohomology of $\mathcal{O}(1, \dotsc, 1)(k 
\alpha)(D_{\mathcal{J}})$ may not vanish even when $\mathcal{J} = 
\{E\}$. However, the following lemma shows that the cohomology only depends on which flats of rank $r-1$ are contained in $\mathcal{J}$. 
\begin{lemma}\label{lem:minus-two-case}
If $\mathcal{J} \subset \mathcal{J}'$ is a containment of order filters 
in the lattice of flats and $\mathcal{J}' \setminus \mathcal{J}$ does not contain any flats of 
rank $r-1$, then for each $i$, the natural map
\[
    H^i(W_L, \mathcal{O}(1, \dotsc, 1)(-2 \alpha)(D_{\mathcal{J}})) \to 
    H^i(W_L, \mathcal{O}(1, \dotsc, 1)(-2\alpha)(D_{\mathcal{J}'}))
\]
is an isomorphism. 
\end{lemma}
\begin{proof}
It suffices to consider the case when $\mathcal{J}' = \mathcal{J} \cup 
\{F\}$ for some flat $F$ whose rank is not $r-1$. By 
\eqref{eq:restrictioncohomology},
$$H^i(D_F, \mathcal{O}_{D_F}(1, \dotsc, 1)(-2 \alpha)(D_{\mathcal{J}'})) =  \bigoplus_{a + b = i}H^{a}(W_{L^F}, \mathcal{O}(1, \dotsc, 1)(-\alpha)) \otimes H^b(\underline{W}_{L_F}, \mathcal{O}(-\alpha)).$$
We have $H^b(\underline{W}_{L_F}, \mathcal{O}(-\alpha)) = H^b(\mathbb{P}L_F, \mathcal{O}(-1))$, which vanishes unless $L_F$ is $1$-dimensional, i.e., $F$ has rank $r-1$. The claim then follows from the long exact sequence of cohomology groups.
\end{proof}

\begin{remark}\label{rmk:minus-two-case}
    In our language, the main theorem in \cite{LenzHierarchical} deals 
    with $H^0(W_L, \OO(1, \ldots, 1)(k\alpha)(D_{\mathcal J}))$ in the 
    case where $k \geq -1$ and $\mathcal{J}$ is any order filter, as well as 
    the case where $k = -2$ and $\mathcal{J}$ 
    contains all flats of rank $r-1$. By 
    \cref{lem:minus-two-case}, the case where $k=-2$ and $\mathcal{J}$ 
    contains all flats of rank $r-1$ is essentially the same as the case $k=-1$. 
    This is not explicitly stated in \cite{LenzHierarchical}, however, it 
    is consistent with \cite[Remark 3.3, Remark 3.19, and Example 
    7.3]{LenzHierarchical}.
\end{remark}

\begin{remark}
Through a modification of the proof of Proposition~\ref{prop:hierarchicalvanishing}, using techniques of \cite{BergetFink} and Theorem~\ref{thm:vanishing}, one can show that for any order filter $\mathcal{J}$,  any $k \ge -1$, any $j \ge 0$, and any $i>0$, we have
$$H^i(W_L, \mathcal{O}(1, \dotsc, 1)(k\alpha)(D_{\mathcal{J}}) \otimes \wedge^j \Omega_L) = 0.$$
\end{remark}

\section{Zonotopal algebras}
We continue to assume that $L \subseteq \F^n$ is a linear subspace of dimension $r$ with 
augmented wonderful variety $W_L$. We now study the sections of line bundles on $W_L$, and in particular we prove Theorem~\ref{thm:powerideal}. 

\begin{definition}
If the restriction of a polynomial $f$ in $\Sym L^*$ to a generic affine line in the direction of $v \in L$ has degree $d$, we say that $f $ has degree $d$ in direction $v$ and write $\deg_v(f) = d$. 
\end{definition}

\begin{definition}
Given a subvariety $Z \subseteq X$ and a rational 
function $f$ on $X$, the order of vanishing of $f$ along $Z$ is
\[
    \ord_Z(f) \coloneqq \min \{ a-b : f=p/q,\ p \in I_Z^a,\ q 
    \in I_Z^b\},
\]
where $I_Z$ is the ideal defining $Z$ in an affine neighborhood of the 
generic point of $Z$. If $\ord_Z(f) \geq 0$, we say $f$ has a zero of order 
$\ord_Z(f)$ at $Z$. If $\ord_Z(f) \leq 0$, we say that $f$ 
has a pole of order $-\ord_Z(f)$ at $Z$.
\end{definition}

We consider $\P L \subseteq \P(L \oplus \C)$ embedded as the subvariety 
where the last homogeneous coordinate is zero. Given a linear subspace 
$U \subseteq L$, we therefore have $\P U \subseteq \P L \subseteq \P(L \oplus \C)$.

\begin{lemma}\label{pole-degree}
    Let $U \subseteq L$ be a linear subspace, and let $f \in 
    \Sym L^*$ be a homogeneous polynomial of degree $k$, which we view as a rational function on $\P(L \oplus \C)$. Then $\ord_{\P 
    U }(f) = -\deg_v(f)$ for a generic $v \in U$.
\end{lemma}
\begin{proof}
    It suffices to consider the case where $L = \C^d$ and $U \subseteq 
    L$ is the $m$-dimensional subspace defined by the vanishing of the 
    last $d-m$ coordinates of $L$. We write
    \[
        f(x_1, \ldots, x_d) = \sum_{\alpha \in A} c_\alpha x^\alpha,\quad 
        A \subseteq \Z_{\geq 0}^d,\ x^{\alpha} := x_1^{\alpha_1} \cdots 
        x_d^{\alpha_d}. 
    \]
    Then $\deg_v(f)$ is the maximum value of $\alpha_1 + \dotsb + \alpha_m$ for some $\alpha \in A$. Let $\beta$ be an element of $A$ which achieves the maximum. We can view $f$ as a function on $\mathbb{C}^d$, so it can be viewed as a rational function on $\mathbb{P}^d$.  We write this rational function 
        as a degree zero rational function in $x_1, \ldots, x_{d+1}$,
    where $x_{d+1}$ is the homogenizing variable:
    \[
        f(\frac{x_1}{x_{d+1}}, \ldots, \frac{x_d}{x_{d+1}}) = 
        \frac{1}{x_{d+1}^k} f(x_1, \ldots, x_d).
    \]
This shows that $\ord_{\P U}(f) = e - k$, where $e$ is the largest 
    integer such that $f \in \langle x_{m+1}, \ldots, x_d, x_{d+1} \rangle^e$. Since there is a monomial in 
    $f$ indexed by $\beta$ such that $\beta_1 + \ldots + \beta_m = 
    \deg_v(f)$, we have 
    that $e \leq k - \deg_v(f)$. Since $\alpha_1 + \ldots + \alpha_m 
    \leq \deg_v(f)$ for all terms in $f$, we have that $e \geq k - 
    \deg_v(f)$.
\end{proof}

\begin{theorem}\label{thm:direction}
    Given integers $a_F$ for each proper flat $F \subseteq E$ of $L$,
    with corresponding subspace $L_F \subseteq L$, 
    \[
        H^0(W_L, \mathcal{O}(\sum a_F D_F)) = \{f \in \Sym L^* 
            : \deg_v(f) \leq a_F,\ F \text{ proper 
        flat},\ v \in L_F\ \text{generic}\}.
    \]
\end{theorem}
\begin{proof}
    Construct the augmented wonderful variety by blowing up $\P L_F 
    \subseteq \P(L \oplus \F)$ for corank one flats $F 
    \subseteq E$, then blowing up the proper transforms of $\P L_F$ for 
    corank two flats, and so on. 
    Suppose that at a given stage of the construction, $\pi \colon X \to \P(L \oplus 
    \F)$ is the composition of all previous blowups, and $Z \subseteq X$ 
    is the center of the next blowup, obtained 
    as the strict transform of $\P L_F \subseteq \P(L \oplus \F)$. By 
    the Proj construction of the blowup \cite[Theorem 22.3.2]{VakilNotes}, for 
    any rational function $f$, $\ord_{Z}(f)$ is equal to the order 
    $\ord_E(f)$ of 
    vanishing of $f$ along the exceptional divisor $E$ of the blowup 
    $\operatorname{Bl}_Z X$. 
    Because $Z$ is the proper transform of $\P L_F$, there is an open 
    neighborhood in $X$ containing the generic point of $Z$ over which 
    the pair $(X,Z)$ is isomorphic to the pair $(\P(L \oplus \F), 
    \P L_F)$. Therefore $\ord_E(f) = \ord_{\P L_F} f$, so we can apply 
    Lemma~\ref{pole-degree} to get $\ord_E(f) = - \deg_v(f)$ for generic 
    $v \in L_F$. Later blowups do not change $\ord_E(f)$ because they 
    occur outside a neighborhood containing the generic point of $E$. 
    The global sections of $\mathcal{O}(\sum a_F 
    D_F)$ consist of rational functions $f$ such that $\ord_{D_F}(f)$ is at least $-a_F$.
\end{proof}

\begin{proof}[Proof of Theorem~\ref{thm:powerideal}]
This follows from Theorem~\ref{thm:direction} and \cite[Proposition 2.6]{ArdilaPostnikov}, which describe the inverse system of a power ideal in terms of restrictions of polynomials to affine lines.  
\end{proof}

\subsection{Deletion-contraction short exact sequences}

\begin{lemma} \label{O-deletion} If $i \in E$ is not a coloop, then the projection $\pi_{E\setminus i} \colon W_{L} 
        \twoheadrightarrow W_{L\setminus i}$ induces an isomorphism 
        \[
            H^0(W_L, \mathcal{O}(1, \ldots, 
            1)(k\alpha)(-W_{L/i}))
                \cong H^0(W_{L\setminus i}, \mathcal{O}(1, \ldots, 1)(k\alpha)).
        \]
\end{lemma}
\begin{proof}
As the divisor class of $W_{L/i}$ is equal to $y_i$ and $i$ is not a coloop, Lemma~\ref{lem:pullback} implies that the line bundle $\mathcal{O}(1, \ldots, 1)(k\alpha)(-W_{L/i})$ is pulled back from $W_{L \setminus i}$. The result follows from the projection formula, as $R\pi_{E \setminus i*} \mathcal{O}_{W_L} = \mathcal{O}_{W_{L \setminus i}}$ by Theorem~\ref{thm:pushforward}. 
\end{proof}

\begin{lemma} \label{L-deletion} If $i \in E$ is not a coloop, then the projection $\pi_{E \setminus i} \colon W_{L} 
        \twoheadrightarrow W_{L\setminus i}$ induces an isomorphism 
        \[
            H^0(W_L, \mathcal{L}_M((k+1)\alpha)(-W_{L/i}))
            \cong H^0(W_{L\setminus i}, \mathcal{L}_{M\setminus i}((k+1)\alpha)).
        \]
\end{lemma}
\begin{proof}
By Theorem~\ref{thm:pushforward}, we have $R\pi_{E \setminus i *} \mathcal{L}_M(-1, \dotsc, -1) = \mathcal{L}_{M \setminus i}(-1, \dotsc, -1)$. The result then follows from the projection formula and Lemma~\ref{lem:pullback}. 
\end{proof}

Now suppose that $i \in E$ is neither a loop nor a coloop, and consider the short exact sequence corresponding to the Cartier divisor $W_{L/i}$: 
\begin{equation} \label{closed-subscheme}
    0 \to \mathcal{O}(-W_{L/i}) \to \mathcal{O}_{W_L} \to 
    \mathcal{O}_{W_{L/i}} \to 0
\end{equation}
We obtain two short exact sequences by tensoring (\ref{closed-subscheme}) with $\mathcal{O}(1, \ldots, 
1)(k\alpha)$ and with the invertible 
subsheaf $\mathcal{L}_M((k+1)\alpha) \subseteq \mathcal{O}(1, 
\ldots, 1)(k\alpha)$. Using Lemma~\ref{O-deletion}, 
Lemma~\ref{L-deletion}, Lemma~\ref{lem:restrictdiv}, Lemma~\ref{lem:restrictLM}, and 
Corollary~\ref{cor:seczonotope}, we get the commutative 
diagram in (\ref{section-diagram}), where the rows are exact.

\begin{equation}\label{section-diagram}
\begin{tikzcd}[column sep = tiny]
0 \arrow[r] & H^0(W_{L\setminus i}, \mathcal{L}_{M\setminus i}((k+1)\alpha)) \arrow[d, hook] \arrow[r]                                          & H^0(W_L, \mathcal{L}_{M}((k+1)\alpha)) \arrow[d, hook] \arrow[r]                                             & H^0(W_{L/i}, \mathcal{L}_{M/i}((k+1)\alpha)) \arrow[d, hook]  \\
0 \arrow[r] & {H^0(W_{L \setminus i}, \mathcal{O}(1, \ldots,1)(k\alpha))} \arrow[r] \arrow[d, "\cong"] & {H^0(W_L, \mathcal{O}(1, \ldots, 
                1)(k\alpha))} \arrow[r] \arrow[d, "\cong"] & {H^0(W_{L/i}, \mathcal{O}(1, \ldots, 
            1)(k\alpha))} \arrow[d, "\cong"]  \\
0 \arrow[r] & {C_{L\setminus i,k}} \arrow[r]                                                                                 & {C_{L,k}} \arrow[r]                                                                                    & {C_{L/i,k}}
\end{tikzcd}
\end{equation}

\begin{proof}[Proof of Proposition~\ref{prop:LM}]
    We show that for all $k \leq -1$, the maps from the top row to the middle row of  
    (\ref{section-diagram}) are isomorphisms.
    First suppose that $k = -1$. Then the middle and top rows are right exact by 
    Corollary~\ref{cor:line} and Corollary~\ref{cor:LMvanish}. If all 
    elements of $E\setminus i$ are loops or coloops, then the left and right vertical arrows 
    are isomorphisms by Example~\ref{LO-boolean}, and so the middle 
    vertical arrow is an isomorphism as well. The case of $k=-1$ and 
    arbitrary $L$ now follows by induction. In 
    general, if $D$ is a Weil divisor on a normal variety $X$, 
    $\alpha$ is an effective Weil divisor not contained in the support 
    of $D$, and $k \leq 0$, then $H^0(X, \mathcal{O}_{X}(D + 
    k\alpha))$ can be viewed as the subspace of $\F(X)$ consisting of those $f \in 
    H^0(X, \mathcal{O}_X(D))$ which vanish to order $-k$ 
    on $\alpha$. So if $D$ and $D'$ do not contain $\alpha$ in 
    their support and $H^0(X, \mathcal{O}_X(D)) 
    = H^0(X, \mathcal{O}_X(D'))$ as subspaces of $\F(X)$, then $H^0(X, \mathcal{O}_X(D 
    + k\alpha)) = H^0(X, \mathcal{O}_X(D' + k\alpha))$. By moving $\alpha$ 
    to be the closure of a generic hyperplane in $L$, we can apply this 
    observation, where $X = W_L$, $\mathcal{O}_X(D) = 
    \mathcal{O}_{W_L}(1, \ldots, 1)(-\alpha)$, and $\mathcal{O}_X(D') = 
    \mathcal{L}_M$. 
\end{proof}
\begin{corollary}\label{middle-exact}
    For all $k \geq -2$, the middle and bottom rows of 
    (\ref{section-diagram}) are right exact.
\end{corollary}
\begin{proof}
    If $k \geq -1$ then exactness follows from Corollary~\ref{cor:line}. 
    If $k = -2$, then the top row is exact by 
    Corollary~\ref{cor:LMvanish}, and the top and middle rows are 
    isomorphic by Proposition~\ref{prop:LM}.
\end{proof}
\begin{example}[The Boolean case]\label{LO-boolean} Suppose that $L = \F^S \subseteq \F^E$ for $S 
    \subseteq E$. Then $W_L$ is the toric variety of the stellahedron on 
    $S$. Loops do not meaningfully change anything we have considered, 
    so we may assume that $S = E$. The inner normal 
    fan of the stellahedron has the standard basis 
    vectors as rays, and the sum of their corresponding divisors is 
    $\mathcal{O}(1, \ldots, 1)$. All other rays are in the negative 
    orthant, and the sum of their corresponding divisors is 
    $\alpha$. It follows from Proposition~\ref{prop:canonical} that $\mathcal{L}_M \cong \mathcal{O}_{W_L}$. 

    On the other hand, consider the 
    sections of $\mathcal{O}(1, \ldots, 
    1)(k\alpha)$. Let $\square = [0,1]^E \subseteq \R^E$ denote the unit 
    cube and let $\Delta = \text{conv}(0, e_i : i \in E) \subseteq 
    \R^E$ denote the standard simplex. If $k \geq 0$ then $\mathcal{O}(1, 
    \ldots, 1)(k\alpha)$ has a section for every lattice point of 
    the Minkowski sum $\square + k \Delta$. Notice that this is never 
    equal to the number of lattice points of $(k+1)\Delta$, which 
    correspond to the sections of $\mathcal{L}_M((k+1)\alpha)$.
    However, if $k \leq 0$, then the sections of 
    $\mathcal{O}(1, \ldots, 1)(k\alpha)$ correspond to lattice 
    points $u \in \square$ such that 
    \[
        \langle u, v_F \rangle \geq -|E\setminus F| - k
    \]
    for all primitive vectors $v_F = -\sum_{i \in E \setminus F} e_i$ 
    along negative rays of the stellahedral fan. If $k = -1$, 
    then the lattice point $0 \in \square$ is the unique 
    such point, so $\mathcal{O}(1, \ldots, 1)(-\alpha)$ has a $1$-dimensional space of sections, and therefore $H^0(W_L, \mathcal{L}_M) \cong 
    H^0(W_L, \mathcal{O}(1, \ldots, 1)(-\alpha))$. If $k < -1$, then 
    both $\mathcal{L}_M((k+1)\alpha)$ and $\mathcal{O}(1, \ldots, 
    1)(k\alpha)$ have no nonzero sections.
\end{example}

\begin{remark}
Most previous proofs of cases of the right exactness of \eqref{eq:deletion} \cite{DM,PSS,ArdilaPostnikov,HoltzRon,BergetTutte, Superspace} proceed by induction on the number of hyperplanes. In order to perform the induction step, one needs to know a spanning set for $C_{L, k}$, which one proves by induction at the same time. This works nicely for $k \ge -1$ (see \cite[Proposition 4.4 and 4.5]{ArdilaPostnikov}), but the method breaks down for $k = -2$. When $k=-2$, Holtz and Ron were able to prove the right exactness of \eqref{eq:deletion} in a different way, but they did not produce an explicit spanning set for $C_{L, -2}$. Holtz and Ron conjectured a spanning set \cite[Conjecture 6.1]{HoltzRon}, but their conjecture was disproved in \cite[Proposition 2]{APCorrection}. A spanning set for $C_{L, -2}$ was eventually found in \cite[Section 4.5]{GillespieThesis}, but it is more complicated, and it still seems challenging to prove the right exactness of \eqref{eq:deletion} using this. 

Our proof of the right exactness of \eqref{eq:deletion} for $k \ge -2$ does not produce an explicit spanning set for $C_{L, k}$. Rather, it can be interpreted as strengthening the induction hypothesis to include the vanishing of certain derived functors. With this stronger induction hypothesis, the right exactness can be easily reduced to the base case, the case of Boolean arrangements. 
\end{remark}

\begin{remark}\label{rem:CP25} Let us compare our point of view to the 
    geometric framework for zonotopal algebras introduced in 
    \cite{GeoZonotopeII}. While we are interested in the augmented 
    wonderful variety $W_L$, \cite{GeoZonotopeII} deals with the 
    \emph{arrangement Schubert variety} $Y_L$, which is defined as the closure of $L$ in $(\mathbb{P}^1)^E$. There is a map $p \colon W_L \to Y_L$ which 
    resolves the singularities of $Y_L$, and satisfies $p^* 
    \mathcal{O}_{Y_L}(1, \ldots, 1) = \mathcal{O}_{W_L}(1, \ldots, 1)$. 
    In addition, the Schubert variety has rational singularities, so 
    $p_* \omega_{W_L} \cong \omega_{Y_L}$ is the canonical sheaf of $Y_L$. Let $D$ be the Weil divisor $Y_L \setminus L$. 
    The authors of 
    \cite{GeoZonotopeII} prove that 
    \begin{equation}\label{CP-interpretation}
        C_{L, k} = H^0(Y_L, \mathcal{O}_{Y_L}(1, \ldots, 1)(kD)) \text{ for }k \le 0,
    \end{equation}
    and that $\mathcal{O}_{Y_L}(-2D) \cong \omega_{Y_L}$ 
    \cite[Theorem 1.14, Proposition 4.4]{GeoZonotopeII}. The equality \eqref{CP-interpretation} does not hold for $k > 0$. Note that $D$ 
    is often not Cartier, so $\mathcal{O}_{Y_L}(1, \ldots, 1)(kD)$ is not an invertible sheaf in general. The right-hand side of 
    \eqref{CP-interpretation} can be easily identified 
    with $H^0(W_L, \mathcal{O}_{W_L}(1, \ldots, 1))$ in the case $k=0$ and with $H^0(W_L, \mathcal{L}_M(-\alpha))$ in the case $k=-2$, as 
    we now explain. If $k=0$, then $H^0(W_L, \mathcal{O}_{W_L}(1, \ldots, 1)) 
    \cong H^0(Y_L, \mathcal{O}_{Y_L}(1, \ldots, 1))$ by the projection 
    formula and the fact that $p_*\mathcal{O}_{W_L} \cong 
    \mathcal{O}_{Y_L}$. If $k = -2$, then $\mathcal{L}_M(-\alpha) \cong \omega_{W_L} \otimes \mathcal{O}_{W_L}(1, \dotsc, 1)$, and the fact that $p_* (\omega_{W_L} \otimes \mathcal{O}_{W_L}(1, \dotsc, 1)) \cong \omega_{Y_L} \otimes \mathcal{O}_{Y_L}(1, \dotsc, 1)$ identifies $H^0(W_L, \mathcal{L}_M(-\alpha))$ with $H^0(Y_L, \mathcal{O}(1, \dotsc, 1)(-2D))$. We do not know a 
    similar connection for $k \not\in \{0,-2\}$. 
\end{remark}

\subsection{Behavior of coloops}

Both for applications to hierarchical zonotopal algebras in Section~\ref{ss:hierarchical} and to prove Theorem~\ref{thm:supersections}, we will need some technical results about the behavior of zonotopal algebras for hyperplane arrangements with coloops. 

We first recall a result of Ardila and Postnikov. A \emph{cocircuit vector} of $L$ is a nonzero element of $L_F$, where $F$ is a flat of rank $r-1$. That is, a cocircuit vector is a generator of a $1$-dimensional flat of the hyperplane arrangement. For $k \le 0$, the ideal of the $k$th zonotopal algebra is generated by powers of cocircuit vectors. 

\begin{proposition}\label{prop:APcocircuit}\cite[Lemma 1]{APCorrection}
For $k \le 0$, the power ideal $(v^{\rho(v) + k + 1} : v \in L \setminus \{0\})$ is generated by $\{v^{\rho(v) + k + 1} : v \text{ is a cocircuit vector}\}$. 
\end{proposition}

\begin{lemma}\label{lem:negkcoloop}
If $k < -1$ and $i$ is a coloop of $L$, then $\mathcal{Z}_{L, k} = 0$. The subspace $\mathcal{C}_{L, -1}$ of $\Psi_{L^*}$ maps isomorphically onto $\mathcal{C}_{L \setminus i, -1}$ under the map $\Psi_{L^*} \to \Psi_{(L \setminus i)^*}$. 
\end{lemma}

\begin{proof}
If $k < -1$, then the cocircuit vector $v$ corresponding to $i$ has $\rho(v) = 1$, so $Z_{L, k} = 0$. As $\mathcal{C}_{L, k}$ is the inverse system of the differential closure of the ideal of $Z_{L, k}$, it is also $0$. For $k=-1$, the elements $v$ and $dv$ belong to the ideal of $\mathcal{Z}_{L, -1}$, and the quotient of $\Psi_L$ by them is identified with $\Psi_{L \setminus i}$. All other cocircuit vectors are contained in the $i$th hyperplane, and so are identified with cocircuit vectors of $L \setminus i$. 
\end{proof}

In our treatment of hierarchical zonotopal algebras, we will also need 
the following vanishing statement.
\begin{lemma}\label{lem:coloop-vanishing}
    If $i \in E$ is a coloop of $L$ and $k$ is any integer, then
    \[
        H^k(W_L, \mathcal{O}(-\alpha + \sum_{j \not= i} y_j)) = 0.
    \]
\end{lemma}
\begin{proof}
    Consider the following short exact sequence of sheaves on $W_L$.
    \[
        0 \to \OO_{W_L}(-\alpha + \sum_{j \not= i} y_j) \to \OO_{W_L}(1, \ldots, 
        1)(-\alpha) \to \OO_{W_{L/i}}(1, \ldots, 1)(-\alpha) \to 0
    \]
    The middle and right sheaves have vanishing higher cohomology 
    by \Cref{cor:line}, and the map on sections between the middle and 
    right line bundles is an isomorphism by 
    \Cref{thm:powerideal,lem:negkcoloop}, so the lemma follows from the 
    long exact sequence in cohomology.
\end{proof}

\subsection{Hierarchical zonotopal algebras}\label{ss:hierarchical}

\begin{definition}[{\cite[Definition 
    3.1]{LenzHierarchical}}]\label{hierarchical-definition}
    Suppose that $\mathcal{J}$ is an order filter in the lattice of flats of $L$ and $k$ is an integer. The hierarchical zonotopal ideal associated to this data is\footnote{Our 
        $I_{L,\mathcal{J},k}$ 
    is equal to $I_{L,\mathcal{J},k+1}$ in 
\cite{LenzHierarchical}.}
\[
    I_{L,\mathcal{J},k} := ( v^{\rho(v) + k + \chi_{\mathcal{J}}(v)+1} : v \in L 
    \setminus \{0\} ) \subseteq \Sym L,\quad \text{where } \chi_{\mathcal{J}}(v) := 
    \begin{cases}
        1, \quad v \in L_F \text{ for some $F \in \mathcal{J}$}\\
        0, \quad \text{otherwise.}
    \end{cases}
\]
\end{definition}

Write $C_{L,\mathcal{J},k} \subseteq \Sym L^*$ for the associated inverse system, 
which is called a hierarchical zonotopal space. As before, given an order filter $\mathcal{J}$ of flats of $L$, let $D_{\mathcal{J}}$ be the divisor 
$\sum_{F \in \mathcal{J} \setminus \{E\}}D_F$. 

\begin{proposition}\label{prop:hierarchical}
    Let $\mathcal{J}$ and $k$ be as in \Cref{hierarchical-definition}. Then the following subspaces of $\Sym L^*$ coincide:
    \[
        C_{L,\mathcal{J},k} = H^0(W_L, \mathcal{O}_{W_L}(1, \ldots, 
        1)(k\alpha)(D_{\mathcal{J}})).
    \]
\end{proposition}
\begin{proof}
The ideal $I_{L,\mathcal{J},k}$ is of the form $(v^{a_F 
+ 1} : v \in L_F \setminus \{0\},\, F \text{ proper flat} )$, where 
$a_F = \rho(v) + k + \chi_{\mathcal{J}}(v)$, so the result follows from 
\Cref{thm:powerideal}.
\end{proof}

We now state the deletion-contraction short exact sequence for 
hierarchical zonotopal spaces that appears in \cite[Proposition 
3.17]{LenzHierarchical}, and we prove that it is exact using our cohomology vanishing results. Recall that a coloop of a flat $F \subseteq E$ is an 
element $i \in F$ such that $F \setminus i$ is also a flat.

\begin{definition}[{\cite[Definition 3.13]{LenzHierarchical}}]
    Let $\mathcal{J}$ be any collection of flats of $L$, and let 
    $i \in E$. Set
    \begin{align*}
        \mathcal{J}\setminus i &:= \{F\setminus i : F \in 
        \mathcal{J},\ \text{$i$ is not a coloop of $F$}\},   \\
            \mathcal{J}/i &:= \{F\setminus i 
            :  F \in \mathcal{J}, \, i \in F\}.
    \end{align*}
    We view $\mathcal{J}\setminus i$ as a collection of flats of 
    $L\setminus i$, and we view $\mathcal{J}/i$ as a collection of flats of 
    $L/i$.
\end{definition}

\begin{lemma}\label{coloop-flats}
Let $F$ be a flat which contains a coloop $i \in F$, and let $\mathcal{J}$ be a collection of flats which contains $\{G : G \ge F\}$ and does not contain any flat $H$ with $H < F$. Then the 
    following inclusion of sheaves induces isomorphisms on cohomology groups.
    \[
        \mathcal{O}_{W_L}(\sum_{j \in E\setminus i} y_j)(k\alpha)(D_{\mathcal{J}}) 
        \subseteq \mathcal{O}_{W_L}(\sum_{j \in E\setminus i} 
        y_j)(k\alpha)(D_{\mathcal{J}} + D_F).
    \]
\end{lemma}
\begin{proof}
    Set $\mathcal{L} = \mathcal{O}_{W_L}(\sum_{j \in E \setminus i} 
y_j)(k\alpha)(D_{\mathcal{J}}+D_F)$ and consider the short exact sequence
    \[
        0 \to \mathcal{L}(-D_F) \to \mathcal{L} \to \mathcal{L} 
        \vert_{D_F} \to 0. 
    \]
    It suffices to show that $\mathcal{L} \vert_{D_F}$ has all 
    cohomology groups equal to zero. 
    By \Cref{lem:restrictToBoundary} and \Cref{lem:restrictdiv}, the first Chern class $c_1(\mathcal{L} \vert_{D_F}) \in 
    H^2(W_{L^F} \times \underline{W}_{L_F})$ is of the form 
    \[
        c_1(\mathcal{L} \vert_{D_F}) = (-\alpha + \sum_{j \in F 
                \setminus i} y_j) \otimes 1 + 1 \otimes z,\quad z \in 
        H^2(\underline{W}_{L_F}).
    \]
    By the K\"unneth formula, the sections of $\mathcal{L} \vert_{D_F}$ have 
    $H^k(W_{L^F}, \mathcal{O}(-\alpha + \sum_{j \in F \setminus i} 
    y_j))$ as a tensor 
factor for some $k$, which is zero by \Cref{lem:coloop-vanishing} since $i$ is a 
coloop of $F$.
\end{proof}

\begin{lemma}\label{lem:hierarchical-pullback-and-restriction}
    Suppose that $\mathcal{J}$ is an order filter of flats of $L$, 
    and that $i \in E$ is neither a loop nor a coloop of $L$. Then the following hold:
\begin{itemize}
\item The pullback of $\mathcal{O}_{W_{L\setminus i}}(1, \ldots, 1)(k\alpha)(D_{\mathcal{J}\setminus i})$ along $\pi_{E\setminus i}$ is a subsheaf of $\mathcal{O}_{W_L}(1, \ldots, 0, \ldots, 1)(k\alpha)(D_\mathcal{J})$, where the zero is in the $i$th spot. 
    \item Moreover, the above inclusion induces isomorphisms on all 
        cohomology groups.
    \item The restriction of $\mathcal{O}_{W_L}(1, \ldots, 
        1)(k\alpha)(D_{\mathcal{J}})$ to $W_{L/i}$ is isomorphic to 
        $\mathcal{O}_{W_{L/i}}(1, \ldots, 
        1)(k\alpha)(D_{\mathcal{J}/i}).$
    \end{itemize}
\end{lemma}
\begin{proof}
    We begin by proving the first statement. 
    By \cite[Section 3]{BHMPW20a}, for any flat $F$ of $L \setminus i$, we have $\pi_{E \setminus i}^* [D_F] = [D_F] + [D_{F \cup i}]$, where we interpret $[D_G]$ as $0$ if $G$ is not a flat of $L$. We see that 
    \[
        \pi_{E\setminus i}^*([D_{\mathcal{J}\setminus i}]) = [D_{\mathcal{J}}] - 
        \sum [D_F],
    \]
    where the sum is over flats $F \in \mathcal{J}$ which have $i$ as a 
    coloop and where $F\setminus i \not\in \mathcal{J}$. Adding 
    $y_1 + \ldots + y_n + k\alpha - y_i$ to both sides of the displayed 
    equation, we see that the pullback of $\mathcal{O}_{W_{L/i}}(1, 
    \ldots, 1)(D_{\mathcal{J}\setminus i})$ plus an effective divisor 
    class is equal to $\mathcal{O}_{W_L}(1, \ldots, 
    0,\ldots,1)(k\alpha)(D_{\mathcal{J}})$, proving the first point. The 
    second point now follows from repeated application of 
    \Cref{coloop-flats}, adding the remaining flats in decreasing order of rank. The third point follows directly from 
    \Cref{lem:restrictdiv} and \Cref{lem:restrictToBoundary}.
\end{proof}

We can now give a cohomological proof of the key proposition in 
\cite{LenzHierarchical}. 

\begin{proposition}[{\cite[Proposition 3.17]{LenzHierarchical}}]
    Suppose that $i \in E$ is neither a loop nor a coloop for $L$, 
    $\mathcal{J}$ is an order filter of flats, and $k \geq 
    -2$. If $k = -2$, assume additionally that $\mathcal{J}$ contains all 
    corank one flats. Then there is a short exact sequence of graded vector spaces
    \[
    0 \to C_{L\setminus i,\mathcal{J}\setminus i,k} \to 
        C_{L,\mathcal{J},k} \to C_{L/i,\mathcal{J}/i,k} \to 0,
    \]
    where the first map is multiplication by a defining equation for 
    $H_i$, and the second map is restriction to $H_i$.
\end{proposition}
\begin{proof}
    By \cref{lem:minus-two-case,rmk:minus-two-case}, the case where 
    $k=-2$ follows from the case where $k=-1$, so assume that $k \geq 
    -1$.
    Set $\mathcal{L} = \mathcal{O}_{W_L}(1, \ldots, 
    1)(k\alpha)(D_{\mathcal{J}})$, and consider the short exact 
    sequence
    \[
        0 \to \mathcal{L}(-W_{L/i}) \to \mathcal{L} \to \mathcal{L} 
        \vert_{W_{L/i}} \to 0.
    \]
    By \Cref{lem:hierarchical-pullback-and-restriction} and \Cref{prop:hierarchical}, the induced long exact sequence in cohomology reads
    \[
    0 \to C_{L\setminus i,\mathcal{J}\setminus i,k} \to 
    C_{L,\mathcal{J},k} \to C_{L/i,\mathcal{J}/i,k} \to 
    H^1(\mathcal{O}_{W_{L \setminus i}}(1, \ldots, 
    1)(k\alpha)(D_{\mathcal{J}\setminus i})) \to \ldots,
    \]
    and the cohomology group on the right vanishes by 
    \Cref{prop:hierarchicalvanishing}.
\end{proof}

\subsection{The truncation short exact sequence}

Fix a generic linear form $f \in L^*$ with zero locus $TL \subseteq L$. The 
matroid associated to the subspace $TL \subseteq \F^E$ is the truncation 
of the matroid associated to $L \subseteq \F^E$. By viewing $W_L$ as the 
closure of $L$ inside the toric variety of the stellahedron, we see 
that the closure of $TL$ in $W_L$ is identified with $W_{TL}$. 
Consider the short exact sequence associated to the Cartier divisor $W_{TL} \subseteq 
W_L$, tensored with $\mathcal{O}_{W_L}(1, \ldots, 1)(k\alpha)$. 
The restriction of $\mathcal{O}_{W_L}(1, \ldots, 1)(k\alpha)$ to $W_{TL}$ is 
$\mathcal{O}_{W_{TL}}(1, \ldots, 1)(k\alpha)$, so we have the short exact sequence

\begin{equation}\label{truncation-SES-sheaves}
\begin{tikzcd}[column sep = small]
0 \arrow[r] & {\mathcal{O}_{W_{L}}(1, \ldots,             
        1)(k\alpha)} \arrow[r] & 
        \mathcal{O}_{W_L}(1, \ldots, 1)((k+1)\alpha)
        \arrow[r]  & 
        \mathcal{O}_{W_{TL}}(1, \ldots, 1)((k+1)\alpha) \arrow[r] & 0.
\end{tikzcd}
\end{equation}

\begin{proposition}\label{negative-k-vanishing}
    We have $H^i(W_L, \mathcal{O}(1, \ldots, 1)(k\alpha)) = 0$ for all $i > 
    \max\{0,-k-1\}$.
\end{proposition}
\begin{proof}
    For $k \geq 0$, this is Corollary~\ref{cor:line}. The case where $k 
    < 0$ follows by induction on $k$ and the dimension of $L$, using the long exact 
    sequence in cohomology associated to (\ref{truncation-SES-sheaves}).
\end{proof}
Of particular interest is the case $k=-2$, where 
Proposition~\ref{negative-k-vanishing} says that 
$H^i(W_L, \mathcal{O}_{W_L}(1, \ldots, 1)(-2\alpha)) = 0$ for $i > 1$. As we 
shall see, the first cohomology need not vanish in this case. Let us 
now describe the sections of the terms in (\ref{truncation-SES-sheaves}). By 
Corollary~\ref{cor:seczonotope} we have the following diagram of 
left exact sequences, where all terms are considered as vector subspaces 
of $\Sym L^*$, the left arrow is multiplication by our generic linear form $f$, and the right 
arrow is restriction to $TL$:
\begin{equation}\label{truncation-diagram}
\begin{tikzcd}[column sep = tiny]
0 \arrow[r] & {H^0(W_L, \mathcal{O}(1, \ldots,             
        1)((k-1)\alpha))} \arrow[r] \arrow[d, "\cong"] & 
        {H^0(W_L, \mathcal{O}(1, \ldots, 1)(k\alpha))} 
        \arrow[r] \arrow[d, "\cong"] & 
        {H^0(W_{TL}, \mathcal{O}(1, \ldots, 1)(k\alpha))} 
        \arrow[d, "\cong"]\\
0 \arrow[r] & C_{L, k-1} \arrow[r] & C_{L, k} \arrow[r]  & C_{TL, k}.
\end{tikzcd}
\end{equation}

\begin{proposition}\label{truncation-corollary}
    For all $k \geq 0$, both of the horizontal arrows on the right of 
    (\ref{truncation-diagram}) are surjective.
\end{proposition}
\begin{proof}
This follows from Corollary~\ref{cor:line}.
\end{proof}

\begin{remark}
    Proposition~\ref{truncation-corollary} can also be proved by induction 
    in the style of \cite{ArdilaPostnikov}, using the generating set that 
    \cite[Proposition 4.5]{ArdilaPostnikov} provides for $C_{L, k}$ when 
    $k \geq -1$.
\end{remark}

\begin{proposition}\label{truncation-sequence-internal}
    There is an exact sequence of graded vector spaces
\begin{center}
\begin{tikzcd}
0 \arrow[r] & {C_{L,-2}^{\bullet - 1}} \arrow[r]                                                                                 
            & {C_{L,-1}^{\bullet}} \arrow[r]                                                                                    
            & {C_{TL,-1}^{\bullet}} \arrow[r] & {H^1(W_L, \mathcal{O}(1, 
\ldots, 1)(-2\alpha))^{\bullet - 1}} \arrow[r] & 0
\end{tikzcd}
\end{center}
\end{proposition}
\begin{proof}
    Take the long exact sequence in cohomology of 
    (\ref{truncation-SES-sheaves}), and apply 
    Proposition~\ref{negative-k-vanishing} and 
    Corollary~\ref{cor:seczonotope}. The grading shift occurs because we must change the linearization of $\mathcal{O}(1, \dotsc, 1)(-2\alpha)$ to make this an exact sequence of $\mathbb{G}_m$-equivariant sheaves. 
\end{proof}

We will use this to compute the Hilbert series of $H^1(W_L, \mathcal{O}(1, \ldots, 1)(-2\alpha))$. 

\begin{proposition}\label{prop:H1hilb}
We have 
$$\operatorname{Hilb}(H^1(W_L, \mathcal{O}(1, \dotsc, 1)(-2\alpha))) = \sum_{A \subseteq E, \, \operatorname{rk}(A) \le r - 2} (-1)^{r - \operatorname{rk}(A)} q^{n - |A| - r + \operatorname{rk}(A)} (1 - q)^{|A| - \operatorname{rk}(A)}.$$
\end{proposition}

\begin{proof}
By Proposition~\ref{truncation-sequence-internal} and known formulas for the Hilbert series of the central and internal zonotopal algebras, we have
\begin{equation*}\begin{split}
q\operatorname{Hilb}(H^1(W_L, \mathcal{O}(1, \dotsc, 1)(-2\alpha))) &= q^{n-r+1} T_L(0, q^{-1}) - q^{n-r} T_L(1, q^{-1}) + q^{n-r+1} T_{TL}(1, q^{-1}).
\end{split}\end{equation*}
By a well-known formula for the Tutte polynomial \cite[(2.13)]{MR1716762}, this is equal to
$$q^{n - r + 1} \sum_{A} (-1)^{r - \operatorname{rk}(A)} (q^{-1} - 1)^{|A| - \operatorname{rk}(A)} - q^{n-r} \sum_{\operatorname{rk}(A) = r} (q^{-1} - 1)^{|A| - r} + q^{n-r + 1} \sum_{\operatorname{rk}(A) \in \{r-1, r\}} (q^{-1} - 1)^{|A| - r + 1}.$$
Grouping the terms associated to a subset $A$ of $E$ together and adding implies the result. 
\end{proof}

In particular, if the coefficient of $q^k$ in the quantity in Proposition~\ref{prop:H1hilb} is $0$ for a hyperplane arrangement $L$, then the sequence in Proposition~\ref{truncation-sequence-internal} is exact in degree $k+1$.

\subsection{Zonotopal algebras for $k < -2$}

We now illustrate how our techniques can be used to analyze zonotopal algebras when $k < -2$. Classically, zonotopal algebras $Z_{L, k}$ for $k < -2$ are difficult to study because the deletion-contraction sequence can fail to be exact (see Example~\ref{ex:MK4}) and, at least for $k \le -6$, the Hilbert series of $Z_{L, k}$ can fail to be combinatorially determined, see \cite[Proposition 4]{APCorrection}. Our results can be applied to these zonotopal algebras in two different ways. 

Firstly, our results provide invariants which always satisfy deletion-contraction identities, and can sometimes be proved to agree with the Hilbert series of a zonotopal algebra. For a (linearized) line bundle $\mathcal{L}$ on $W_L$, the $\mathbb{G}_m$-equivariant Euler characteristic is defined by
$$\chi_{\mathbb{G}_m}(W_L, \mathcal{L}) \coloneqq \operatorname{Hilb}(H^0(W_L, \mathcal{L})) - \operatorname{Hilb}(H^1(W_L, \mathcal{L})) + \dotsb + (-1)^r \operatorname{Hilb}(H^r(W_L, \mathcal{L})).$$
If the higher cohomology of $\mathcal{L}$ vanishes, then $\chi_{\mathbb{G}_m}(W_L, \mathcal{L})$ agrees with the Hilbert series of the inverse system of the corresponding power ideal. More generally, if the Hilbert series of $H^i(W_L, \mathcal{L})$ vanishes in a particular degree for each $i \ge 1$, then $\chi_{\mathbb{G}_m}(W_L, \mathcal{L})$ computes the dimension of the corresponding graded piece of the inverse system. In cases of interest, the $\mathbb{G}_m$-equivariant Euler characteristic satisfies a deletion-contraction recurrence, meaning that it can be computed from the case of Boolean arrangements. 

\begin{proposition}\label{prop:euleradd}
If $i$ is not a loop or a coloop, then for any $k$,
$$\chi_{\mathbb{G}_m}(W_L, \mathcal{O}(1, \dotsc, 1)(k\alpha)) = \chi_{\mathbb{G}_m}(W_{L/i}, \mathcal{O}(1, \dotsc, 1)(k \alpha)) + q \chi_{\mathbb{G}_m}(W_{L \setminus i}, \mathcal{O}(1, \dotsc, 1)(k\alpha)) \text{, and}$$
$$\chi_{\mathbb{G}_m}(W_L, \mathcal{L}_M((k+1)\alpha)) = \chi_{\mathbb{G}_m}(W_{L/i}, \mathcal{L}_{M/i}((k+1)\alpha)) + q \chi_{\mathbb{G}_m}(W_{L \setminus i}, \mathcal{L}_{M \setminus i}((k+1)\alpha)).$$
\end{proposition}
Together with the cohomology vanishing results (Corollary~\ref{cor:line} and Corollary~\ref{cor:LMvanish}), Proposition~\ref{prop:euleradd} directly implies the known formulas for the Hilbert series of $Z_{L, k}$ when $k \ge -2$ \cite{ArdilaPostnikov,HoltzRon}. 

\begin{proof}[Proof of Proposition~\ref{prop:euleradd}]
We do the case of $\mathcal{O}(1, \dotsc, 1)(k\alpha)$; the case of $\mathcal{L}_M((k+1)\alpha)$ is similar. Because $i$ is not a loop, there is a short exact sequence of sheaves
$$0 \to \mathcal{O}(1, \dotsc, 1)(k\alpha)(-W_{L/i}) \to \mathcal{O}(1, \dotsc, 1)(k\alpha) \to \mathcal{O}_{W_{L/i}}(1, \dotsc, 1)(k \alpha) \to 0.$$
This is a short exact sequence of $\mathbb{G}_m$-equivariant sheaves, as long as $\mathcal{O}(1, \dotsc, 1)(k\alpha)(-W_{L/i})$ is given the appropriate linearization (obtained by tensoring the usual linearization with the linearization of $\mathcal{O}$ where $\operatorname{Hilb}(H^0(W_L, \mathcal{O})) = q$). 
As the $\mathbb{G}_m$-equivariant Euler characteristic is additive in short exact sequences,
$$\chi_{\mathbb{G}_m}(W_L, \mathcal{O}(1, \dotsc, 1)(k\alpha)) = \chi_{\mathbb{G}_m}(W_{L/i}, \mathcal{O}(1, \dotsc, 1)(k\alpha)) + \chi_{\mathbb{G}_m}(W_{L}, \mathcal{O}(1, \dotsc, 1)(k\alpha)(-W_{L/i})).$$
Because $i$ is not a coloop, $\mathcal{O}(1, \dotsc, 1)(k\alpha)(-W_{L/i})$ is the pullback of $\mathcal{O}_{W_{L \setminus i}}(1, \dotsc, 1)(k\alpha)$ along the deletion map $\pi_{E \setminus i} \colon W_L \to W_{L \setminus i}$. By the projection formula and Theorem~\ref{thm:pushforward}, $\chi_{\mathbb{G}_m}(W_L, \mathcal{O}(1, \dotsc, 1)(k\alpha)(-W_{L/i})) = q \chi_{\mathbb{G}_m}(W_{L \setminus i}, \mathcal{O}(1, \dotsc, 1)(k\alpha))$, giving the result. 
\end{proof}

When $k \le -2$, $\mathcal{O}(1, \dotsc, 1)(k\alpha)$ typically has a lot of higher cohomology.  Proposition~\ref{prop:H1hilb} shows that $H^1(W_L, \mathcal{O}(1, \dotsc, 1)(-2\alpha))$ is usually large. Even for Boolean arrangements, the cohomology of $\mathcal{O}(1, \dotsc, 1)(k\alpha)$ can be somewhat complicated. For example, using Proposition~\ref{prop:H1hilb} and \eqref{truncation-SES-sheaves}, one can show that if $L = \mathbb{C}^n$ is a Boolean arrangement with $n > 0$, then
$$\operatorname{Hilb}(H^1(W_L, \mathcal{O}(1, \dotsc, 1)(-3 \alpha))) = n q^{-1}, \, \operatorname{Hilb}(H^2(W_L, \mathcal{O}(1, \dotsc, 1)(-3 \alpha))) = \binom{n-1}{2},$$
and $H^i(W_L, \mathcal{O}(1, \dotsc, 1)(-3 \alpha)) = 0$ for $i \not \in \{1, 2\}$. 

The line bundles $\mathcal{L}_M((k+1)\alpha)$ tend to have significantly less cohomology, but they can still be used to analyze $Z_{L, k}$ when $k \le -1$. When $L = \mathbb{C}^n$ is Boolean, the line bundle $\mathcal{L}_M$ is trivial, and so the cohomology of $\mathcal{L}_M((k+1)\alpha)$ on $W_L$ is the same as the cohomology of $\mathcal{O}(k+1)$ on $\mathbb{P}^{r}$. 

For example, when $k = -3$, the cases when $\mathcal{L}_M(-2\alpha)$ has nonzero cohomology on the augmented wonderful variety of a Boolean arrangement are when $n = 0$ (then $\operatorname{Hilb}(H^0(W_L, \mathcal{L}_M(-2\alpha))) = 1$, and all other cohomology groups vanish) and when $n=1$ (then $\operatorname{Hilb}(H^1(W_L, \mathcal{L}_M(-2\alpha))) = q^{-1}$, and all other cohomology groups vanish). In particular, we have
$$\chi_{\mathbb{G}_m}(W_{\mathbb{C}^n}, \mathcal{L}_M(-2\alpha)) = \begin{cases} 1 & n=0 \\ -q^{-1} & n=1 \\ 0 & n > 1.\end{cases}$$
For a general arrangement $L$, $\chi_{\mathbb{G}_m}(W_L, \mathcal{L}_M(-2\alpha))$ is an evaluation of the ``umbral Tutte polynomial'' of Ardila and Postnikov \cite[Proposition 4.11]{ArdilaPostnikov}. 

\medskip

The other way in which our results can be applied to analyze zonotopal algebras when $k < -2$ is by trying to directly compute the cohomology groups of $\mathcal{O}(1, \dotsc, 1)(k \alpha)$ or $\mathcal{L}_M((k+1)\alpha)$. As discussed above, $\mathcal{L}_M((k+1)\alpha)$ typically has less cohomology, so this seems more promising. We illustrate this in one example when $k = -3$ and the  deletion-contraction sequence is not exact. 

As a preliminary step, we compute the case of a collection of generic hyperplanes in a $1$ or $2$-dimensional vector space. We discussed the case of Boolean arrangements above. If $\dim L = 1$ and $n=2$, then the deletion-contraction exact sequence is 
$$0 \to H^0(W_L, \mathcal{L}_M(-2\alpha)) \to \mathbb{C} \to \mathbb{C} \to H^1(W_L, \mathcal{L}_M(-2\alpha)) \to 0,$$
where each copy of $\mathbb{C}$ is in degree $0$. We cannot use this sequence to compute the cohomology, but $W_L \cong \mathbb{P}^1$ and the line bundle $\mathcal{L}_M(-2\alpha)$ is isomorphic to $\mathcal{O}(-1)$ in this case, so its cohomology vanishes. For $\dim L = 1$ and $n > 2$, we have $H^1(W_{L \setminus i}, \mathcal{L}_{M \setminus i}(-2\alpha)) = 0$ by induction. This implies that the deletion-contraction sequence is right exact, and so we see that
$$\operatorname{Hilb}(H^0(W_L, \mathcal{L}_M(-2\alpha))) = 1 + q + \dotsb + q^{n-3} \text{ and }H^i(W_L, \mathcal{L}_M(-2\alpha)) = 0 \text{ for }i>0.$$
We now consider the case when we have $n > 2$ generic hyperplanes in a vector space $L$ of dimension $2$. If $n = 3$, then the cohomology of $\mathcal{L}_{M \setminus i}(-2\alpha)$ on $W_{L \setminus i}$ and $\mathcal{L}_{M/i}(-2\alpha)$ on $W_{L/i}$ both vanish for any $i$, so the cohomology of $\mathcal{L}_M(-2\alpha)$ vanishes. If $n > 3$, then the deletion-contraction sequence is exact, and we see by induction that
$$\operatorname{Hilb}(H^0(W_L, \mathcal{L}_M(-2\alpha))) = 1 + 2q + 3q^2 + \dotsb + (n-3)q^{n-4}\text{ and }H^i(W_L, \mathcal{L}_M(-2\alpha)) = 0 \text{ for }i>0.$$

\begin{example}\label{ex:MK4}
Let $L$ be the hyperplane arrangement given as 
$$L = \operatorname{row \,span }\begin{pmatrix} 1 & 1 & 0 & 0 & 1 \\ 0 & 0 & 1 & 1 & 1\end{pmatrix}.$$
We will use the deletion-contraction sequence for the $5$th hyperplane to compute the cohomology of $\mathcal{L}_M(-2\alpha)$ on $W_L$. The arrangement $L/5$ consists of $4$ nonzero hyperplanes in a $1$-dimensional vector space, so, as computed above, we have $\operatorname{Hilb}(H^0(W_{L/5}, \mathcal{L}_M(-2\alpha))) = 1 + q$ and the higher cohomology vanishes.

The arrangement $L \setminus 5$ is a direct sum of two arrangements which each consist of two hyperplanes in a $1$-dimensional space. Applying the deletion-contraction sequence twice and using that deleting a loop does not change the augmented wonderful variety, we see that $\operatorname{Hilb}(H^1(W_{L \setminus 5}, \mathcal{L}_M(-2\alpha))) = 1$ and $H^i(W_{L \setminus 5}, \mathcal{L}_M(-2\alpha)) = 0$ for all other $i$.

 The deletion-contraction sequence then looks like
$$0 \to H^0(W_{L}, \mathcal{L}_M(-2\alpha)) \to \mathbb{C} \oplus \mathbb{C} \to 0 \oplus \mathbb{C} \to H^1(W_{L}, \mathcal{L}_M(-2\alpha)) \to 0,$$
where the Hilbert function is indicated using the direct sum structure, i.e., the Hilbert function of the third term in the exact sequence is $1 + q$ and the Hilbert function of the fourth term is $q$. From this sequence alone, we cannot determine the cohomology of $\mathcal{L}_M(-2\alpha)$, i.e., we cannot see whether the boundary map vanishes. However, a direct computation shows that $\operatorname{Hilb}(Z_{L, -3}) = 1 $, so the boundary map is nonzero and $\operatorname{Hilb}(H^1(W_{L}, \mathcal{L}_{M}(-2\alpha))) = 0$. 
\end{example}
\begin{figure}
\[
\begin{array}{|c|c|c|}
\hline
 & \operatorname{Hilb}(H^0) & \operatorname{Hilb}(H^1) \\
\hline
L/5 & 1 + q & 0 \\ 
\hline 
L \setminus 5 & 0& 1\\ 
\hline
L & 1  & 0 \\ 
\hline
\end{array}
\]\caption{The Hilbert series of the cohomology groups computed in Example~\ref{ex:MK4}.}
\end{figure}

Example~\ref{ex:MK4} also shows that the sequence $0 \to \mathcal{C}_{L \setminus 5, -2}^{\bullet - 1, \bullet} \to \mathcal{C}_{L, -2}^{\bullet, \bullet} \to \mathcal{C}_{L/5, -2}^{\bullet, \bullet} \oplus \mathcal{C}_{L/5, -2}^{\bullet, \bullet - 1}$ constructed in Theorem~\ref{thm:superspaceexact} need not be right exact, because, by Proposition~\ref{prop:LM}, the top anticommutative graded piece of this sequence is identified with the sequence $0 \to C_{L \setminus 5, -3}^{\bullet - 1} \to C_{L, -3}^{\bullet} \to C_{L/5, -3}^{\bullet}$.

\subsection{Polymatroidal power ideals} 
In this section, we describe a family of hyperplane arrangement power ideals with nice properties. Let $f \colon 2^E \to \mathbb{Z}$ be a function on the set of subsets of $E$. We say that $f$ is a \emph{polymatroid} on $E$ if 
\begin{enumerate}
\item $f(\emptyset) = 0$, 
\item for any $S \subseteq T$, we have $f(S) \le f(T)$, and
\item for any $S, T$, we have $f(S \cap T) + f(S \cup T) \le f(S) + f(T)$. 
\end{enumerate}
For a polymatroid $f$, consider the divisor $D_f = \sum_{F \not= E} f(E \setminus F) D_F$ on $W_L$, where the sum is over proper flats of our hyperplane arrangement. If $L = \mathbb{C}^n$, i.e., if our arrangement is Boolean, then the $[D_f]$ are exactly the nef divisor classes on the stellahedral toric variety \cite[Proposition 3.13]{EHL}. When $f$ is the rank function of the dual matroid of the matroid of $L$, then $D_f = c_1(\mathcal{L}_M)$. 
Let $x_1, \dotsc, x_n$ be the coordinate functions on $L$.  We have the following result.

\begin{proposition}\label{prop:polymatroid}
Let $f$ be a polymatroid on $E$. Then the Hilbert series of the ideal $I_f = (v^{f(E \setminus F) + 1} : v \in L_F \setminus \{0\})$ depends only on $f$ and the matroid of $L$. Furthermore, the inverse system of $I_f$ is spanned by polynomials of the form $x_1^{a_1} \dotsb x_n^{a_n}$ where $(a_1, \dotsc, a_n)$ has the property that $a_i \ge 0$ for each $i$, and $\sum_{i \in S} a_i \le f(S)$ for each $S \subseteq E$. 
\end{proposition}

Proposition~\ref{prop:polymatroid} applies to the line bundle $\mathcal{O}(1, \dotsc, 1)(k\alpha)$ and $\mathcal{L}_M(k\alpha)$ for $k \ge 0$. However, it in general does not apply to $\mathcal{O}(1, \dotsc, 1)(k\alpha)$ or $\mathcal{L}_M(k\alpha)$ for $k < 0$. 

\begin{proof}[Proof of Proposition~\ref{prop:polymatroid}]
The Hilbert series of $I_f$ is determined by the Hilbert series of its inverse system.
By \cite[Theorem 5.5]{EFL}, we have $H^i(W_L, \mathcal{O}_{W_L}(D_f)) = 0$ for $i > 0$. This implies that the Hilbert series of the inverse system of $I_f$, which is equal to $H^0(W_L, \mathcal{O}_{W_L}(D_f))$ by Theorem~\ref{thm:powerideal}, is equal to the $\mathbb{G}_m$-equivariant Euler characteristic of the line bundle $\mathcal{O}_{W_L}(D_f)$. It follows from the proof of \cite[Corollary 5.4]{EHL} that this $\mathbb{G}_m$-equivariant Euler characteristic only depends on the matroid of $L$ and $f$. There is an inclusion of $W_L$ into the stellahedral toric variety $X_n$, and $\mathcal{O}_{W_L}(D_f)$ is the restriction of the line bundle $\mathcal{O}_{X_n}(D_f)$. As $\mathcal{O}_{X_n}(D_f)$ is a nef line bundle on a toric variety, its sections are in bijection with the lattice points in a certain polytope. These lattice points are exactly the points $(a_1, \dotsc, a_n)$ satisfying the conditions in the proposition, see \cite[Section 3.3]{EHL}. By \cite[Theorem 5.5]{EFL}, the restriction map $H^0(X_n, \mathcal{O}_{X_n}(D_f)) \to H^0(W_L, \mathcal{O}_{W_L}(D_f))$ is surjective, giving the spanning set for the inverse system. 
\end{proof}

In \cite{LLPP}, the Euler characteristics of line bundles on augmented wonderful varieties were studied. In \cite[Theorem 7.2]{LLPP}, a general formula for the Euler characteristic of a line bundle was given. Using this formula, one can obtain a formula for the total dimension of the inverse system $I_f^{\perp}$ for any polymatroid $f$. This formula is rather complicated, but in some cases it is possible to use tools developed in \cite{LLPP} to obtain simpler formulas. For example, we check that $\chi(W_L, \mathcal{L}_M) = T_L(1,1)$ using this strategy, which also follows from the fact that $\chi(W_L, \mathcal{L}_M) = \dim H^0(W_L, \mathcal{L}_M) = \dim C_{L, -1}$. The exceptional Hirzebruch--Riemann--Roch theorem of \cite[Theorem 1.6]{LLPP} allows us to compute this Euler characteristic in terms of intersection theory on $W_L$. Using \cite[Corollary 6.5]{EHL} and \cite[Proposition 5.3]{LLPP}, we have that
$$\chi(W_L, \mathcal{L}_M) = \int_{W_L} \frac{s(\mathcal{Q}_L^*)}{1 - \alpha},$$
where $s(\mathcal{Q}_L^{*})$ is the total Segre class of the conormal bundle of $W_L$ in the stellahedral toric variety $X_n$. By \cite[Theorem 1.11 and Corollary 5.11(1)]{EHL}, this intersection number is equal to the coefficient of $w^{n-r}$ in $(1 + w)^{n-r} T_L(1, 1 + (1 + w)^{-1})$, which is equal to $T_L(1,1)$.

\begin{remark}\label{rem:BV}
In \cite{BrionVerge}, Brion and Vergne identify $C_{L, -1}^{n-r}$, the $(n-r)$th graded piece of the inverse system of the central zonotopal algebra, with $H^{r}_{dR}(L \setminus \cup H_i; \F)$, the top graded piece of the de Rham cohomology of the hyperplane arrangement complement $L \setminus \cup H_i$. This can be recovered using our results, as follows. For any polymatroid $f$, it follows from \cite[Theorem A]{EFL} that the restriction map $H^0(W_L, \mathcal{O}_{W_L}(D_f)) \to H^0(\underline{W}_L, \mathcal{O}_{W_L}(D_f)|_{\underline{W}_L})$ maps the $f(E)$th graded piece isomorphically onto $H^0(\underline{W}_L, \mathcal{O}_{W_L}(D_f)|_{\underline{W}_L})$. If $f$ is the rank function of the matroid dual of the matroid of our hyperplane arrangement, then $\mathcal{O}_{W_L}(D_f) \cong \mathcal{L}_M$, and the restriction of $\mathcal{L}_M$ to $\underline{W}_L$ is the log-canonical bundle of $\underline{W}_L$ (viewed as a compactification of the projectivized arrangement complement). Because the de Rham cohomology of the projectivized arrangement complement is endowed with a mixed Hodge structure which, in degree $i$, is pure of weight $2i$, the sections of the log-canonical bundle on $\underline{W}_L$ are isomorphic to the degree $(r-1)$ part of the de Rham cohomology of the projectivized arrangement complement, which can be identified with $H^{r}_{dR}(L \setminus \cup H_i; \F)$. 
\end{remark}

\section{Superspace preliminaries}\label{sec:superspace}
Given a vector space $L$, recall that the \emph{superspace ring} 
$\Psi_{L}$ is the ring of algebraic differential forms with polynomial 
coefficients on the dual vector space $L^*$:
\[
    \Psi_L \coloneqq \Sym L \otimes \bigwedge^\bullet L.
\]
We have seen that there is an action of $\Sym L$ on $\Sym L^*$ where $v 
\in L$ acts by its directional derivative $D_v$. Likewise, there is an 
action of $\wedge^\bullet L$ on $\wedge^\bullet L^*$, where $v \in 
L$ acts by contraction: for $\omega \in \wedge^i L^*$, set $\iota_{v}(\omega)$ to be the element $\omega(v \wedge -) \in \Hom(\wedge^{i-1} L, \F)$, which is naturally identified with $\wedge^{i-1} L^*$.

There is an induced action of $\Psi_L$ on $\Psi_{L^*}$.
Choose dual bases $e_1, \ldots, e_r \in L$ and $x_1, \ldots, x_r \in 
L^*$. Given an element $v \in L$, we will write $v \coloneqq v \otimes 1 \in 
\Psi_L$ for a commutative generator and $dv \coloneqq 1 \otimes v \in 
\Psi_L$ for an anticommutative generator.
In these coordinates, an element $f(e_1, \ldots, e_r, de_1, \ldots, de_r) 
\in \Psi_L$ acts by the operator $f(\frac{\partial}{\partial x_1}, 
\ldots, \frac{\partial}{\partial x_r}, \iota_{e_1}, \ldots, 
\iota_{e_r})$ on $\Psi_{L^*}$. 

\begin{example}
    Set $L = \F^2$, $f = e_1de_2 - e_2de_1$, and $g = x_1dx_1\wedge dx_2$. Then
    \begin{align*}
        f\odot g &= (\frac{\partial}{\partial x_1} \circ 
        \iota_{e_2} - \frac{\partial}{\partial x_2}\circ 
        \iota_{e_1}) \odot (x_1dx_1\wedge dx_2)\\
                 &= (\frac{\partial}{\partial x_1} 
                 x_1)(\iota_{e_2} dx_1\wedge dx_2) - (\frac{\partial}{\partial x_2} 
                 x_1)(\iota_{e_1} dx_1 \wedge dx_2)\\
                 &= -dx_1.
        \end{align*}
\end{example}
Another example is that $\iota_{e_1} \iota_{e_2} dx_1 dx_2 = -1$. We now recall a variation on the version of inverse systems for $\Psi_L$ introduced in \cite[Section 2]{Superspace}. 

\begin{lemma}
    There is a perfect pairing $\Psi_L \times \Psi_{L^*} \to 
    \F$ which sends $(f,g)$ to the constant term of $f\odot g$.
\end{lemma}
\begin{proof}
    Given monomials $\alpha,\beta,\gamma,\delta \in \C[u_1, \ldots, 
    u_r]$, one can check that the constant term of 
    \[
        \alpha(\frac{\partial}{\partial x_1}, \ldots, \frac{\partial}{\partial 
        x_r}) \beta(\iota_{e_1}, \ldots, \iota_{e_r}) \odot 
        \gamma(x_1, \ldots, x_r)\delta(de_1, \ldots, de_r)
    \]
    is nonzero exactly when $\alpha=\gamma$ and $\beta=\delta$.
    The $\alpha\beta$ factor and the $\gamma\delta$ factor in the above 
    expression form bases for $\Psi_L$ and $\Psi_{L^*}$, respectively, so 
    the result follows.
\end{proof}

\begin{lemma}\label{mult-adjoint}
    For a homogeneous element $f \in \Psi_L^{a,b}$, the adjoint of the multiplication action $f \times -:\Psi_L \to \Psi_L$ is given 
    component-wise by $(-1)^{b(d-b)} f \odot -:\Psi_{L^*}^{c,d} \to 
    \Psi_{L^*}^{c-a,d-b}$.
\end{lemma}
\begin{proof}
    Suppose that $g \in \Psi_{L^*}^{c,d}$, and $h \in \Psi_{L}^{c-a,d-b}$. 
    We have that
    \[
        fh \odot g = (-1)^{b(d-b)}hf\odot g = (-1)^{b(d-b)}h \odot (f \odot g),
    \]
    so the left and right expressions have the same 
    constant term. If $h \in \Psi_{L}$ is homogeneous of degree other 
    than $(c-a,d-b)$, then the left and right expressions both have zero constant term.
\end{proof}

Given a bigraded ideal $I \subseteq \Psi_L$, we define the inverse system as
\[
    I^\perp \coloneqq \{g \in \Psi_{L^*} : f\odot g = 0 \text{ for all 
    $f \in I$}\}.
\]
Because $I$ is an ideal, an element $g \in \Psi_{L^*}$ lies in $I^{\perp}$ if and only if $f \odot g$ has $0$ constant term for all $f \in I$. 
This gives a natural identification of $I^\perp$ with $(\Psi_L/I)^*$ as bigraded 
vector spaces.
The exterior derivative $d$ is a linear map $\Psi_L \to \Psi_L$ of bidegree $(-1, 1)$ which, in terms of the chosen basis, satisfies
$$d f(e_1, \dotsc, e_r) de_{i_1} \dotsb de_{i_k} = \sum_{j=1}^{r} \frac{\partial f}{\partial e_j} de_{j} de_{i_1} \dotsb de_{i_k}.$$
\begin{remark}
    Our convention differs slightly from that of \cite{Superspace}, 
    where the authors write $\theta_i \coloneqq dx_i$, and define the total 
    derivative as the operator $\sum_{i = 1}^r 
    \frac{\partial}{\partial x_i} \theta_i$, which differs by $-1$ 
    from the operator $d$ above when applied to an element of odd
    anticommutative degree $k$. 
\end{remark}

For $v \in L$, set $\tau_v$ to the map from $L$ to $L$ given by $\tau_v(w) = v + w$. 
This induces an action of the additive group of $L$ on $\Psi_{L^*}$ by pullback, 
which we denote by $\tau_v^*$. In coordinates, for $v = (v_1, \dotsc, v_r)$, we have
\[
    \tau_v^*(x_{i_1}\cdots x_{i_a} dx_{j_1} \cdots  dx_{j_b}) = (x_{i_1} + 
    v_{i_1})\cdots(x_{i_a}+v_{i_a}) dx_{j_1} \cdots dx_{j_b}.
\]

\begin{lemma}
    A bigraded vector subspace $S \subseteq \Psi_{L^*}$ is the inverse system 
    of a bigraded ideal $I \subseteq \Psi_L$ if and only if $\tau_v^* S 
    \subseteq S$ and $\iota_v S \subseteq S$ for all $v \in L$.
\end{lemma}
\begin{proof}
Lemma~\ref{mult-adjoint} implies that $S$ is the inverse system of a bigraded ideal if and only if $S$ is closed under $\frac{\partial}{\partial x_i}$ and $\iota_{e_i}$ for each $i$. 
It suffices to check that $S$ is closed under $\frac{\partial}{\partial x_i}$ for each $i$ if and only if it is closed under $\tau_{e_i}^*$ for each $i$. If $S$ is closed under $\tau_{e_i}^*$, then for each $f \in S$, $\tau_{e_i}^*(f) - f$ lies in $S$. We have
    \[
        \tau_{e_i}^*(f) - f = \frac{\partial f}{\partial x_i} + 
        \text{lower order terms},
    \]
    which implies that $\frac{\partial f}{\partial x_i} \in S$ because 
    $S$ is bigraded. Conversely, if $S$ is closed under partial differentiation, then it follows from Taylor expansion that $S$ is closed under the translation action. 
\end{proof}

Let us consider the special element $\mathcal{E} \in L^* \otimes L$, which 
we call the \emph{Euler vector field}:
\[
    \mathcal{E} \coloneqq x_1e_1 + \ldots + x_re_r.
\]
We can also describe $\mathcal{E}$ as the vector field given by the 
infinitesimal action of the group $\F^*$ acting on $L$ by scaling.
This description shows that $\mathcal{E}$ does not depend on our choice of basis.
The \emph{contraction} $\iota_{\mathcal{E}}\colon \Psi_{L^*} \to 
\Psi_{L^*}$ is an operator of bidegree $(1,-1)$, defined by
\[
    \iota_{\mathcal{E}} = \sum_{i = 1}^r x_i \iota_{e_i}.
\]
The operator $\iota_{\mathcal{E}}$ is also called the interior 
derivative (with respect to $\mathcal{E}$). The exterior and interior 
derivatives are related by Cartan's magic formula:
\[
    (d\iota_{\mathcal{E}} + \iota_{\mathcal{E}}d) \omega = (a+b)\omega,\quad 
    \omega \in \Psi_{L^*}^{a,b}.
\]
Cartan's magic formula has the following elementary consequence.
\begin{lemma}\label{lem:diffclosedexact}
    Suppose that $I \subseteq \Psi_L$ is an ideal generated by homogeneous polynomials
    $f_1, \ldots, f_k \in \Sym L$ together with their exterior derivatives 
    $df_1, \ldots, df_k \in \Psi_L$. Then $dI \subseteq I$ and 
    $\iota_{\mathcal{E}} I \subseteq I$. Moreover, the derivative $d \colon \Psi_L/I \to \Psi_L/I$ is exact except in degree $(0,0)$.
\end{lemma}
\begin{proof}
    Both $d$ and $\iota_{\mathcal{E}}$ satisfy a Leibniz rule for homogeneous elements:
    \[
        d(fg) = (df)g \pm f(dg)\quad \text{and} \quad 
        \iota_{\mathcal{E}}(fg) = (\iota_{\mathcal{E}} f)g \pm 
        f(\iota_{\mathcal{E}} g).
    \]
    Therefore to prove that $dI \subseteq I$ and $\iota_{\mathcal{E}}I 
    \subseteq I$, it is enough to check that $d(f_i), d(df_i), 
    \iota_{\mathcal{E}}(f_i),$ and $\iota_{\mathcal{E}}(df_i)$ all lie in $I$. The only 
    nontrivial one is the last, which follows from Cartan's magic 
    formula. Indeed, suppose that $f_i$ has degree $a$. Then
    \[
        \iota_{\mathcal{E}}(df_i) = a f_i - d\iota_{\mathcal{E}}f_i 
    = a f_i.
    \]
    Therefore both $d$ and $\iota_{\mathcal{E}}$ descend to operators on 
    $\Psi_{L}/I$. To prove the exactness statement, suppose that $f \in 
    (\Psi_L/I)^{a,b}$ is a class with $df = 0$. Then by Cartan's magic formula,
    \[
        d\iota_{\mathcal{E}} f = (d\iota_{\mathcal{E}} + \iota_{\mathcal{E}}d)f = (a+b)f.
    \]
    Therefore, as long as $(a+b) \not= 0$, $f = d 
    (\frac{\iota_{\mathcal{E}}f}{a+b})$, so $f$ is in the image of $d$. 
\end{proof}

The operators $d$ and $\iota_{\mathcal{E}}$ are related in another way 
as well:
\begin{lemma}
    The adjoint of $d \colon \Psi_L \to \Psi_L$ 
    is given component-wise by $(-1)^{b-1}\iota_{\mathcal{E}} \colon 
    \Psi_{L^*}^{a,b} \to \Psi_{L^*}^{a+1,b-1}$. 
\end{lemma}

This can be checked on a monomial basis for $\Psi_L$ and $\Psi_{L^*}$. 
We have the following corollary.
\begin{corollary}\label{d-closed-inverse-char}
    A bigraded vector subspace $S \subseteq \Psi_{L^*}$ is the 
inverse system of a bigraded differentially closed ideal $I \subseteq 
\Psi_L$ if and only if $\tau_v^* S \subseteq S$ and $\iota_v S \subseteq 
S$ for all $v \in L$, and $\iota_{\mathcal{E}} S \subseteq S$.
\end{corollary}

From now on, we will no longer focus on differentially closed ideals in 
$\Psi_L$, and we will instead think about their inverse systems 
directly using Corollary~\ref{d-closed-inverse-char}.

\begin{lemma}\label{lem:eulertangent}
    The Euler vector field $\mathcal{E}$ on $L$ extends to a vector field on 
    the augmented wonderful variety $W_L$ which is tangent to the 
    boundary, and therefore gives a section of the log tangent bundle $\Omega_{L}^*$.
\end{lemma}
\begin{proof}
    The vector field $\mathcal{E}$ is the infinitesimal action of the 
    Lie group $\F^*$ acting on $L$ by scaling. Since the $\F^*$ action 
    extends to $W_L$ and preserves the boundary, the result follows.
\end{proof}

\begin{proof}[Proof of Theorem~\ref{thm:diffclosed}]
Let $S$ be the subspace of $\Psi_{L^*}$ obtained by restricting $\bigoplus_j H^0(W_L, \mathcal{L} \otimes \wedge^j \Omega_L)$ to $L$. 
By Corollary~\ref{d-closed-inverse-char}, we need to check that for any $v \in L$, we have $\tau_v^* S \subseteq S$, $\iota_v S \subseteq S$, and also that $\iota_{\mathcal{E}} S \subseteq S$. 

The line bundle $\mathcal{L}$ is isomorphic to $\mathcal{O}(D)$ for a unique divisor $D$ which does not intersect $L$. 
This gives a canonical trivialization of the restriction of $\mathcal{L}$ to $L$, and $\mathcal{L}$ has a canonical $\mathbb{G}_m \ltimes L$-linearization, which is characterized by the property that it restricts to the usual $\mathbb{G}_m \ltimes L$ linearization on $\mathcal{O}_L$. The bundle $\wedge^j \Omega_L$ also has a canonical $\mathbb{G}_m \ltimes L$-linearization. This implies that $S$ is a bigraded subspace of $\Psi_{L^*}$ and $S$ is closed under the action of $L$, i.e., $\tau_v^* S \subseteq S$ for all $v \in L$. 

The action of $L$ on $W_L$ induces a map from $L$, viewed as the Lie algebra of the additive group of $L$, to the space of logarithmic vector fields on $W_L$. Contracting with the logarithmic vector field corresponding to $v \in L$ induces a map $\mathcal{L} \otimes \wedge^j \Omega_L \to \mathcal{L} \otimes \wedge^{j-1} \Omega_L$. The induced map on global sections is $\iota_v$, so $\iota_v S \subseteq S$.

By Lemma~\ref{lem:eulertangent}, the Euler vector field defines a section of the  log-tangent 
bundle $\Omega_{W_L}^*$. Therefore contracting with the Euler vector field 
takes a section of $\wedge^j \Omega_{W_L}$ to a section of 
$\wedge^{j-1} \Omega_{W_L}$. By the Leibniz rule for 
$\iota_{\mathcal{E}}$ and the fact that it annihilates the sections of $\mathcal{L}$, $S$ is closed under $\iota_{\mathcal{E}}$.
\end{proof}

\section{Superspace zonotopal algebras}

In this section, we will study inverse systems of the form $\bigoplus_j H^0(W_L, \mathcal{O}(1, \dotsc, 1)(k\alpha) \otimes \wedge^j \Omega_L)$ and prove Theorems~\ref{thm:supersections}, \ref{thm:superspaceexact}, and \ref{thm:exact}. 
We begin the proof of Theorem~\ref{thm:supersections} by verifying one inclusion. 

\begin{proposition}\label{prop:containment}
For any $k \le 0$, $\bigoplus_j H^0(W_L, \mathcal{O}(1, \dotsc, 1)(k\alpha) \otimes \wedge^j \Omega_L)$ is contained in $\mathcal{C}_{L, k}$. 
\end{proposition}

\begin{proof}
Let $s$ be a section of $\mathcal{O}(1, \dotsc, 1)(k\alpha) \otimes \wedge^j \Omega_L$, and let $v$ be a cocircuit vector. Let $F$ denote the flat of rank $r-1$ which has $v \in L_F$. By Proposition~\ref{prop:APcocircuit}, it suffices to show that
\begin{equation}\label{eq:inversesystem}
v^{\rho(v) + k + 1} \odot s = 0 \text{ and }d(v^{\rho(v) + k + 1}) \odot s = 0,
\end{equation}
as the ideal of $\mathcal{Z}_{L, k}$ is generated by $\{v^{\rho(v) + k + 1}, d(v^{\rho(v) + k + 1}) : v \text{ is a cocircuit vector}\},$ see \cite[(1.2)]{Superspace}. To do this, we will analyze the behavior of $s$ along the divisor $D_F$. 

The divisor $D_F$ can be identified with $W_{L^F} \times \underline{W}_{L_F}$. Because $\dim L_F = 1$, $\underline{W}_{L_F}$ is a point, so $D_F$ is identified with $W_{L^F}$. Let $D_F^{\circ}$ denote the open subset of $W_{L^F}$ which is identified with $L^F$. We view $D_F^{\circ}$ as a subset of $W_L$. Choose an identification of $L^F$ with $\mathbb{C}^{r-1}$, and choose a splitting $L \cong L_F \oplus \mathbb{C}^{r-1}$. This induces an isomorphism of $L \sqcup D_F^{\circ}$ with $\mathbb{P}(L_F \oplus \mathbb{C}) \times \mathbb{C}^{r-1}$, where $D_F^{\circ}$ is identified with $\{\infty\} \times \mathbb{C}^{r-1}$. 

The restriction of $\mathcal{O}(1, \dotsc, 1)(k\alpha)$ to $L \sqcup D_F^{\circ}$ is the line bundle $\mathcal{O}((\rho(v) + k)D_F^{\circ})$, and the restriction of $\Omega_L$ is the log cotangent bundle of $\mathbb{P}^1 \times \mathbb{C}^{r-1}$, viewed as a partial compactification of $L$. We can write the restriction of $s$ to $L$ as 
$$s|_{L} = \sum_{i_1, \dotsc, i_{j-1}} f_{i_1, \dotsc, i_{j-1}}(v, z_1, \dotsc, z_{r-1}) dv \wedge dz_{i_1} \wedge \dotsb \wedge dz_{i_{j-1}} + \tau,$$ where $z_1, \dotsc, z_{r-1}$ is a basis for $\mathbb{C}^{r-1}$, $\tau$ is a differential form which is pulled back from $\mathbb{C}^{r-1}$, and $f_{i_1, \dotsc, i_{j-1}}$ is a polynomial. The fact that $s$ is a section of the tensor product of $\mathcal{O}((\rho(v) + k)D_F^{\circ})$ with the $j$th exterior power of the log cotangent bundle implies that each $f_{i_1, \dotsc, i_{j-1}}$ has degree at most $\rho(v) + k - 1$ in $v$, so \eqref{eq:inversesystem} holds for degree reasons. 
\end{proof}

\begin{proposition}\label{prop:exactsuperspace}
For some $i$ which is not a loop or coloop of $L$, choose a splitting of the inclusion $L/i \hookrightarrow L$.
For any $k$, there is an exact sequence
\begin{equation*}\begin{split}
0 \to & \bigoplus_j H^0(W_{L \setminus i}, \mathcal{O}(1, \dotsc, 1)(k\alpha) \otimes \wedge^j \Omega_{L \setminus i}) \to \bigoplus_j H^0(W_L, \mathcal{O}(1, \dotsc, 1)(k\alpha) \otimes \wedge^j \Omega_L) \to \\ 
 &\bigoplus_j H^0(W_{L/i}, \mathcal{O}(1, \dotsc, 1)(k\alpha) \otimes (\wedge^j \Omega_{L/i} \oplus \wedge^{j-1} \Omega_{L/i})).
\end{split}\end{equation*}
If $k  \ge -1$, then this sequence is right exact. 
\end{proposition}

\begin{proof}
For ease of notation, we assume that $i = 1$. 
By tensoring the short exact sequence $0 \to \mathcal{O}(-W_{L/1}) \to \mathcal{O}_{W_L} \to \mathcal{O}_{W_{L/1}} \to 0$ with $\mathcal{O}(1, \dotsc, 1)(k\alpha) \otimes \wedge^j \Omega_L$, Lemmas~\ref{lem:restriction} and \ref{lem:split} give a short exact sequence of sheaves
$$0 \to \mathcal{O}(0, \dotsc, 1)(k \alpha) \otimes \wedge^j \Omega_L \to \mathcal{O}(1, \dotsc, 1)(k \alpha) \otimes \wedge^j \Omega_L \to \mathcal{O}_{W_{L/1}}(1, \dotsc, 1)(k\alpha) \otimes (\wedge^j \Omega_{L/1} \oplus \wedge^{j-1}\Omega_{L/1}) \to 0.$$
Let $\pi \colon W_L \to W_{L \setminus 1}$ be the deletion map. 
Because $i$ is not a coloop, $\mathcal{O}(0, \dotsc, 1)(k\alpha) = \pi^* \mathcal{O}(1, \dotsc, 1)(k\alpha)$ by Lemma~\ref{lem:pullback}. By Theorem~\ref{thm:pushforward}, $R \pi_* \wedge^j \Omega_L = \wedge^j \Omega_{L \setminus i}$. The projection formula then gives an isomorphism
\begin{equation}\label{eq:projdel}
H^0(W_L, \mathcal{O}(0, \dotsc, 1)(k\alpha) \otimes \wedge^j \Omega_L) \cong H^0(W_{L \setminus i}, \mathcal{O}(1, \dotsc, 1)(k\alpha) \otimes \wedge^j \Omega_{L \setminus i}).
\end{equation}
Taking global sections then produces the desired left exact sequence. To prove that this sequence is right exact for $k \ge -1$, it suffices to note that, by Theorem~\ref{thm:vanishing}, 
\begin{equation*}
H^1(W_L, \mathcal{O}(0, \dotsc, 1)(k\alpha) \otimes \wedge^j \Omega_L) \cong H^1(W_{L \setminus i}, \mathcal{O}(1, \dotsc, 1)(k\alpha) \otimes \wedge^j \Omega_{L \setminus i}) = 0. \qedhere
\end{equation*}
\end{proof}

For later use, we describe the maps in Proposition~\ref{prop:exactsuperspace}. As mentioned previously, we can view all of these spaces of sections as differential forms on $L$ or on $L/i$. 
The map $H^0(W_{L \setminus i}, \mathcal{O}(1, \dotsc, 1)(k\alpha) \otimes \wedge^j \Omega_{L \setminus i}) \to H^0(W_{L}, \mathcal{O}(1, \dotsc, 1)(k\alpha) \otimes \wedge^j \Omega_L)$ is given by multiplication by a linear form whose vanishing locus is the $i$th hyperplane. As such, it increases the commutative degree by $1$ and preserves the anticommutative degree. The map $H^0(W_L, \mathcal{O}(1, \dotsc, 1)(k\alpha) \otimes \wedge^j \Omega_L) \to H^0(W_{L/i}, \mathcal{O}(1, \dotsc, 1)(k \alpha) \otimes \wedge^j \Omega_{L/i})$ is given by restricting the differential form to $L/i$, and so it preserves the commutative degree and the anticommutative degree. The map $H^0(W_L, \mathcal{O}(1, \dotsc, 1)(k\alpha) \otimes \wedge^j \Omega_L) \to H^0(W_{L/i}, \mathcal{O}(1, \dotsc, 1)(k \alpha) \otimes \wedge^{j-1} \Omega_{L/i})$ is given by contracting the differential form with a generator of the kernel of the chosen splitting $L \to L/i$ and then restricting to $L/i$, and so it preserves the commutative degree and decreases the anticommutative degree by $1$. 

We first prove the special case $k=0$ of Theorem~\ref{thm:supersections}. This relies on the easier part of the result of Rhoades, Tewari, and Wilson: that the sequence in Theorem~\ref{thm:superspaceexact} is left exact when $k=0$. 

\begin{proposition}\label{prop:inversek=0}
As subspaces of $\Psi_{L^*}$, $\mathcal{C}_{L, 0} =  \bigoplus_j H^0(W_L, \mathcal{O}(1, \dotsc, 1) \otimes \wedge^j \Omega_L)$. 
\end{proposition}

\begin{proof}
Proposition~\ref{prop:containment} shows that $\bigoplus_j H^0(W_L, \mathcal{O}(1, \dotsc, 1) \otimes \wedge^j \Omega_L)$ is contained in $\mathcal{C}_{L, 0}$. If $L$ is Boolean, then it is easy to check that this containment is an equality, for example by comparing \eqref{eq:booleanposk} to \cite[Lemma 3.6]{Superspace}. 
For a general hyperplane arrangement, we can delete elements until we obtain the Boolean arrangement of the same rank. If $i$ is a loop, then $W_L$ is identified with $W_{L \setminus i}$, so 
$$\bigoplus_j H^0(W_L, \mathcal{O}(1, \dotsc, 1) \otimes \wedge^j \Omega_L) = \bigoplus_j H^0(W_{L \setminus i}, \mathcal{O}(1, \dotsc, 1) \otimes \wedge^j \Omega_{L \setminus i})\text{ and }\mathcal{C}_{L, 0} = \mathcal{C}_{L \setminus i, 0}.$$
If $i$ is neither a loop nor a coloop, then, using the easier part of \cite[Lemma 3.10]{Superspace}, the choice of a splitting of the inclusion $L/i \hookrightarrow L$ induces an exact sequence 
$$0 \to \mathcal{C}_{L \setminus i, 0} \to \mathcal{C}_{L, 0} \to \mathcal{C}_{L/i, 0} \oplus \mathcal{C}_{L/i, 0}.$$
The maps in this sequence are the same as the maps in Proposition~\ref{prop:exactsuperspace}. 
By induction, the containments for the left term and right term are equalities. The result then follows from the snake lemma. 
\end{proof}

\begin{proof}[Proof of Theorem~\ref{thm:supersections}]
We induct on $-k$; the case $k=0$ is Proposition~\ref{prop:inversek=0}, so we may assume that $k \le -1$. We first consider the case when $L$ has a coloop $i$. If $k < -1$, Lemma~\ref{lem:negkcoloop} gives that $\mathcal{C}_{L, k} = 0$, so Proposition~\ref{prop:containment} implies the result. If $k = -1$, then Lemma~\ref{lem:negkcoloop} identifies $\mathcal{C}_{L, -1}$ with $\mathcal{C}_{L \setminus i, -1}$, so we must identify $H^0(W_L, \mathcal{O}(1, \dotsc, 1)(-\alpha) \otimes \wedge^j \Omega_L)$ with $H^0(W_{L \setminus i}, \mathcal{O}(1, \dotsc, 1)(-\alpha) \otimes \wedge^j \Omega_{L \setminus i})$. 

By Proposition~\ref{prop:SD}, there is an isomorphism
$$ H^0(W_L, \mathcal{O}(1, \dotsc, 1)(-\alpha) \otimes \wedge^j \Omega_L)^* \cong H^r(W_L, \mathcal{O}(-1, \dotsc, -1) \otimes \wedge^{r-j} \Omega_L).$$
Because $i$ is a coloop, there is a birational map $p \colon W_L \to W_{L \setminus i} \times \mathbb{P}^1$, and $\Omega_L$ is isomorphic to $p^*(\Omega_{L \setminus i} \boxplus \mathcal{O}_{\mathbb{P}^1}(-1))$, so $ \wedge^{r-j} \Omega_L \cong p^*( \wedge^{r-j} \Omega_{L \setminus i} \boxtimes \mathcal{O}_{\mathbb{P}^1}) \oplus p^*(\wedge^{r-j-1} \Omega_{L \setminus i} \boxtimes \mathcal{O}_{\mathbb{P}^1}(-1))$. 
The projection formula and the K\"{u}nneth formula then give that 
$$H^r(W_L, \mathcal{O}(-1, \dotsc, -1) \otimes \wedge^{r-j} \Omega_L) \cong H^{r-1}(W_{L \setminus i}, \mathcal{O}(-1, \dotsc, -1) \otimes \wedge^{r-j-1} \Omega_{L \setminus i}) \otimes H^1(\mathbb{P}^1, \mathcal{O}(-2)).$$
As $\dim H^1(\mathbb{P}^1, \mathcal{O}(-2)) = 1$, applying Proposition~\ref{prop:SD} on $W_{L \setminus i}$ gives the desired identification. This isomorphism is compatible with the equality in Lemma~\ref{lem:negkcoloop}, which is induced by the identification $\Psi_{L^*} = \Psi_{(L \setminus i)^*} \otimes \Psi_{\mathbb{C}}$. 

We may therefore assume that $L$ has no coloops. Suppose we have an element $x \in \mathcal{C}_{L, k}^{\bullet, j}$ for some $k \le -1$. We view $x$ as a differential $j$-form on $L$, and we will show that $x$ defines a section of $\mathcal{O}(1, \dotsc, 1)(k\alpha) \otimes \wedge^j \Omega_L$. 
We will then be done by Proposition~\ref{prop:containment}. 

For each $i \in E$, because $i$ is not a coloop, $\pi_{E \setminus i}$ identifies the vector space $L$ with the vector space $L \setminus i$. We claim that $x$ lies in $\mathcal{C}_{L \setminus i, k+1}^{\bullet, j}$ for all $i$. By Proposition~\ref{prop:APcocircuit}, the ideal of $\mathcal{Z}_{L \setminus i, k+1}$ is generated by $\{v^{\rho_{L \setminus i}(v) + k + 2}, \, d(v^{\rho_{L \setminus i}(v) + k + 2}) : v \text{ is a cocircuit vector of }L \setminus i\}$, where $\rho_{L \setminus i}(v)$ is the number of hyperplanes in $L \setminus i$ which do not contain $v$. We need to check that
$$v^{\rho_{L \setminus i}(v) + k + 2} \odot x = 0 \text{ and }d(v^{\rho_{L \setminus i}(v) + k + 2}) \odot x = 0.$$
But because $x \in \mathcal{C}_{L, k}$, we have
$v^{\rho(v) + k + 1} \odot x = 0$ and $d(v^{\rho(v) + k + 1}) \odot x = 0.$
The claim follows because $\rho(v) \le \rho_{L \setminus i}(v) + 1$. 

From the claim above, induction, and the projection formula (as in \eqref{eq:projdel}), we deduce that $x$ lies in $H^0(W_L, \mathcal{O}(1, \dotsc, 1)((k+1)\alpha)(-y_i) \otimes \wedge^{j} \Omega_L)$ for any $i$. Viewed as differential $j$-forms on $L$, we have that
\begin{equation}\label{eq:intersectsections}
H^0(W_L, \mathcal{O}(1, \dotsc, 1)(k\alpha) \otimes \wedge^{j} \Omega_L) = \bigcap_{i} H^0(W_L, \mathcal{O}(1, \dotsc, 1)((k+1)\alpha)(-y_i) \otimes \wedge^{j} \Omega_L).
\end{equation}
Indeed, we have $y_i = \sum_{F \not \ni i} [D_F]$ and $\alpha = \sum_F [D_F]$ in $H^2(W_L)$. Realizing $\mathcal{O}(-\alpha)$ and each $\mathcal{O}(-y_i)$ as the line bundle associated to a Cartier divisor which is contained in the boundary, we have $\mathcal{O}(-\alpha) = \cap_i \mathcal{O}(-y_i)$ as subsheaves of $\mathcal{O}$. Tensoring with the locally free sheaf $\mathcal{O}(1, \dotsc, 1)((k+1)\alpha) \otimes \wedge^{j} \Omega_L$ and then taking global sections gives \eqref{eq:intersectsections}. 
\end{proof}

\begin{proof}[Proof of Theorem~\ref{thm:superspaceexact}]
This follows from combining Theorem~\ref{thm:supersections} and Proposition~\ref{prop:exactsuperspace}. 
\end{proof}

\begin{proof}[Proof of Corollary~\ref{cor:hilbert}]
Taking into account the degree shifts, Theorem~\ref{thm:superspaceexact} shows that if $i$ is  not a loop or coloop, then
$$\operatorname{Hilb}(\mathcal{Z}_{L, -1}) = q \operatorname{Hilb}(\mathcal{Z}_{L \setminus i, -1}) + (1 + t) \operatorname{Hilb}(\mathcal{Z}_{L/i, -1}).$$
If $i$ is a loop, then $\mathcal{Z}_{L, -1} = \mathcal{Z}_{L \setminus i, -1}$, and so $\operatorname{Hilb}(\mathcal{Z}_{L, -1}) = \operatorname{Hilb}(\mathcal{Z}_{L \setminus i, -1})$. If $i$ is a coloop, then $\operatorname{Hilb}(\mathcal{Z}_{L, -1}) = \operatorname{Hilb}(\mathcal{Z}_{L \setminus i, -1})$ by Lemma~\ref{lem:negkcoloop}. 
Note that if $\dim L = 0$, then $\operatorname{Hilb}(\mathcal{Z}_{L, -1}) = 1$. The conclusion then follows from a well-known property of the Tutte polynomial, see, e.g., \cite[(2.5)]{MR1716762}. 
\end{proof}

We now begin the proof of Theorem~\ref{thm:exact}. As mentioned in Lemma~\ref{lem:eulertangent}, the Euler vector field defines a section $s_{\mathcal{E}}$ of the log tangent bundle $\Omega_L^*$ of $W_L$. We give an alternative description of this section. Recall that $W_L$ is embedded in the stellahedral toric variety $X_n$, which is the augmented wonderful variety of the Boolean arrangement $\mathbb{C}^n$. On $X_n$, there is a subbundle $\mathcal{S}_L$ of $\oplus \pi_i^* \mathcal{O}(1)$ which restricts to the inclusion $\Omega_L^* \subset \oplus \pi_i^* \mathcal{O}(1)$, the dual of which was mentioned in Section~\ref{ssec:logcotangent}. There is a distinguished section $s \in H^0(X_n, \oplus \pi_i^* \mathcal{O}(1))$ whose restriction to $\mathbb{C}^n \subset X_n$ is identified with the section $(x_1, \dotsc, x_n)$ of $\mathcal{O}^{\oplus n}$ under a certain distinguished trivialization of the restriction of $\oplus \pi_i^* \mathcal{O}(1)$ to $\mathbb{C}^n$, see \cite[Section 5]{EHL}. The restriction of $s$ to $H^0(W_L, \oplus \pi_i^* \mathcal{O}(1))$ is the image of $s_{\mathcal{E}}$ under the inclusion $H^0(W_L, \Omega_L^*) \subseteq H^0(W_L, \oplus \pi_i^* \mathcal{O}(1))$. 

\begin{lemma}\label{lem:eulervanish}
The vanishing locus of $s_{\mathcal{E}}$ is the origin of $L$.
\end{lemma}

\begin{proof}
The vanishing locus of $s_{\mathcal{E}}$ does not change whether we view it in $\Omega_L^*$ or in $\oplus \pi_i^* \mathcal{O}(1)$. But the vanishing locus of $s$ in $X_n$ already consists of only the origin, as $s$ is the pullback of a section of a vector bundle on $(\mathbb{P}^1)^n$ with that property. 
\end{proof}

\begin{proof}[Proof of Theorem~\ref{thm:exact}]
By Lemma~\ref{lem:eulervanish}, we have a Koszul resolution of the structure sheaf of the origin $\mathcal{O}_o$:
$$0 \to \wedge^r \Omega_L \to \dotsb \to \Omega_L \to \mathcal{O}_{W_L} \to \mathcal{O}_o \to 0.$$
Here the differential is given by contracting with the Euler vector field. If we tensor this complex with $\mathcal{O}(1, \dotsc, 1)(k\alpha)$ for any $k \ge -1$, then it remains exact and the terms have no higher cohomology by Theorem~\ref{thm:vanishing}. We therefore obtain an exact sequence
$$0 \to H^0(W_L, \mathcal{O}(1, \dotsc, 1)(k\alpha) \otimes \wedge^r \Omega_L) \to \dotsb  \to H^0(W_L, \mathcal{O}(1, \dotsc, 1)(k\alpha)) \to \mathbb{C} \to 0,$$
see \cite[Proposition B.1.1]{Lazarsfeld1}. As $d$ is the dual of contraction with the Euler vector field, we obtain the result by dualizing this exact sequence. 
\end{proof}

\subsection{Superspace zonotopal algebras for positive $k$}\label{ssec:positivek}

In \cite[Section 6.2]{Superspace}, Rhoades, Tewari, and Wilson note that the differential closures of the ideals of $Z_{L, k}$ do not seem to behave well for $k \ge 1$, even in the case of Boolean arrangements. See \cite[Example 6.3]{Superspace}. We now discuss how the results above show that, for each $k \ge 1$, there is a well-behaved differentially closed ideal in $\Psi_L$ whose intersection with $\operatorname{Sym} L$ is the ideal of $Z_{L, k}$. Unfortunately, we do not know generators for this ideal (or its inverse system). 

For each $k$, let $\mathcal{B}_{L, k}$ be $\bigoplus_j H^0(W_L, \mathcal{O}(1, \dotsc, 1)(k \alpha) \otimes \wedge^j \Omega_L)$, viewed as a subspace of $\Psi_{L^*}$, and let
$\mathcal{W}_{L, k}$ be the quotient of $\Psi_L$ whose inverse system is $\mathcal{B}_{L,k}$. 
When $k \le 0$, Theorem~\ref{thm:supersections} states that $\mathcal{W}_{L,k}$ is equal to $\mathcal{Z}_{L, k}$, the quotient of $\Psi_L$ by the differential closure of the ideal of $Z_{L, k}$. 

When $k \ge 1$, $\mathcal{W}_{L, k}$ is a quotient of $\mathcal{Z}_{L, k}$. Indeed, the intersection of the ideal of $\mathcal{W}_{L, k}$ with $\operatorname{Sym} L$ is the ideal whose inverse system is $H^0(W_L, \mathcal{O}(1, \dotsc, 1)(k\alpha))$, which is equal to $Z_{L, k}^*$ by Corollary~\ref{cor:seczonotope}. In particular, as the ideal of $\mathcal{W}_{L, k}$ contains the ideal of $Z_{L, k}$ and is differentially closed by Theorem~\ref{thm:diffclosed}, it contains the ideal of $\mathcal{Z}_{L, k}$. 
However, $\mathcal{W}_{L,k}$ is rarely equal to $\mathcal{Z}_{L, k}$ when $k \ge 1$, even in the case of the rank $2$ Boolean arrangement \cite[Example 6.3]{Superspace}. 

We now note several properties of $\mathcal{W}_{L, k}$ for $k \ge 1$ which generalize the attractive properties of $\mathcal{Z}_{L, k}$ for $k \in \{-1, 0\}$. The following result is a direct corollary of Proposition~\ref{prop:exactsuperspace}. 

\begin{corollary}
For any $k \ge -1$ and $i$ which is not a coloop of $L$, there is an exact sequence
$$0 \to \mathcal{B}_{L \setminus i, k}^{\bullet - 1, \bullet} \to \mathcal{B}_{L, k}^{\bullet, \bullet} \to \mathcal{B}_{L/i, k}^{\bullet, \bullet} \oplus \mathcal{B}_{L/i, k}^{\bullet, \bullet - 1} \to 0.$$
\end{corollary}

As $\mathcal{W}_{L, k}$ is the quotient of $\Psi_L$ by a differentially closed ideal (by Theorem~\ref{thm:diffclosed}), there is map $d \colon \mathcal{W}_{L, k} \to \mathcal{W}_{L, k}$ of bidegree $(-1, 1)$. This map satisfies $d^2 = 0$. Then Theorem~\ref{thm:exact} gives the following. 

\begin{corollary}
For any $k \ge -1$, $d$ is exact on $\mathcal{W}_{L, k}$ except in bidegree $(0,0)$. 
\end{corollary}

When $L$ is Boolean, we can give the following description of $\mathcal{B}_{L, k}$ for $k \ge 0$. When $L$ is Boolean, $\Omega_L$ is identified with  $\oplus \pi_i^* \mathcal{O}(-1)$, and so $\mathcal{O}(1, \dotsc, 1) \otimes \wedge^j \Omega_L \stackrel{\sim}{\to} \oplus_{|S| = j} \pi_{E \setminus S} \mathcal{O}(1, \dotsc, 1)$. Then
\begin{equation}\label{eq:booleanposk}
\mathcal{B}_{L, k} = \bigoplus_{S \subseteq E} \bigoplus_{\substack{a_1 + a_2 + \dotsb + a_n \le k \\ T \subseteq S^c}} \mathbb{C} \cdot x_1^{a_1} \dotsb x_n^{a_n} x^T dx_S,
\end{equation}
where $x^T = \prod_{i \in T} x_i$, and $dx_S = dx_{i_1} \wedge \dotsb 
\wedge dx_{i_s}$, where $S = \{i_1, \dotsc, i_s\}$. This agrees with \cite[Lemma 3.6]{Superspace} when $k=0$. It would be interesting to explicitly describe the ideal of $\mathcal{W}_{L, k}$ in general.

\subsection{Applications}

In \cite[Theorem 6.6]{BrylawskiEquality}, Brylawski classified linear relations which hold for the coefficients of the Tutte polynomials of all matroids. As we will explain, all of the nontrivial equalities are consequences of Corollary~\ref{cor:hilbert} and Theorem~\ref{thm:exact}. 

\begin{proposition}\label{prop:brylawksiequalities}
Let $t_{i,j}$ be the coefficient of $x^iy^j$ in the Tutte polynomial of $L$. Then for each $1 \le k \le n-1$,
$$\sum_{i=0}^{k} (-1)^{i} \sum_{u=0}^{r-k+i} \binom{r-u}{k-i} t_{u, n-r-i} = 0.$$
\end{proposition}

After an upper triangular change of variables, this is equivalent to the equalities in \cite[Theorem 6.6]{BrylawskiEquality}. 

\begin{proof}
By Corollary~\ref{cor:hilbert}, we have that
$$\operatorname{Hilb}(\mathcal{Z}_{L, -1}) = \sum_{u,v} t_{u,v}(1 + t)^{r - u} q^{n - r - v}.$$
This implies that $\dim \mathcal{Z}_{L, -1}^{i, k-i} =  \sum_{u=0}^{r-k+i} \binom{r-u}{k-i} t_{u, n-r-i}$. By Theorem~\ref{thm:exact}, there is an exact complex
$$0 \to \mathcal{Z}_{L, -1}^{k, 0} \to \mathcal{Z}_{L, -1}^{k-1, 1} \to \dotsb \to \mathcal{Z}_{L, -1}^{0, k} \to 0.$$
This implies the $\sum_{i=0}^{k} (-1)^i \dim \mathcal{Z}_{L, -1}^{k-i, i} = 0$, giving the result. 
\end{proof}

\begin{remark}
Proposition~\ref{prop:brylawksiequalities} proves the linear relations among the coefficients of the Tutte polynomial for matroids which are realizable over $\mathbb{C}$. Because the coefficients of the Tutte polynomial are valuative invariants of matroids (see \cite{DF10}), a standard argument using \cite[Theorem 5.4]{DF10} implies that these relations hold for all matroids. See \cite[Lemma 5.9]{BEST}. 
\end{remark}

We also note the following interesting example, which was shown to us by Vasu Tewari.

\begin{example}\label{ex:braid}
When $L$ is the $(n-1)$-dimensional braid arrangement, $\mathcal{Z}_{L, -1}$ is equipped with an action of $S_n$. By Corollary~\ref{cor:hilbert}, the Hilbert series of the top commutative graded piece of $\mathcal{Z}_{L, -1}$, namely $\mathcal{Z}^{\binom{n-1}{2}, \bullet}_{L, -1}$, is 
$$(t+1)^{n-1} T_{L}(\frac{1}{t+1}, 0 ) = \prod_{i=1}^{n-2} (i(t+1) + 1).$$
In particular, the total dimension of $\mathcal{Z}^{\binom{n-1}{2}, \bullet}_{L, -1}$ is $(2n-3)!!$, which is the number of rooted binary trees with $n$ labeled leaves. The set of such trees has a natural $S_n$ action.  By Remark~\ref{rem:BV}, $\mathcal{Z}_{L, -1}^{\binom{n-1}{2}, 0}$ is identified with the dual of the top graded piece of the Orlik--Solomon algebra of the braid arrangement. The top graded piece of the Orlik--Solomon algebra of the $(n-1)$-dimensional braid arrangement is identified with the top graded piece of the singular cohomology ring of $M_{0, n+1}$, so it has an action of $S_{n+1}$. It is known that the top graded piece of the Orlik--Solomon algebra of the braid arrangement is the Whitehouse representation of $S_{n+1}$, see \cite{KeelTevelev}. The restriction of this representation to $S_n$ is the Lie representation of $S_n$, see \cite{WhitehouseRep}. 
The dimension of $\mathcal{Z}^{\binom{n-1}{2}, n-2}_{L, -1}$ is $(n-2)!$, the dimension of the Whitehouse representation of $S_n$. 
By Theorem~\ref{thm:exact}, $d$ induces an isomorphism from  $\mathcal{Z}^{\binom{n-1}{2}, n-2}_{L, -1}$ to  $\mathcal{Z}^{\binom{n-1}{2} - 1, n-1}_{L, -1}$, which is identified with $Z^{\binom{n-1}{2}-1}_{L, -2}$, i.e., to the top graded piece of the internal zonotopal algebra of the braid arrangement. 
It would be interesting to compute the $S_n$ representation on $\mathcal{Z}_{L, -1}^{\binom{n-1}{2}, \bullet}$. 
In \cite{GeoZonotopeI}, the authors constructed an $S_n$-equivariant isomorphism between the internal zonotopal algebra of the \emph{dual} of the $(n-1)$-dimensional braid arrangement and the cohomology of the configuration space of $(n-1)$ points in $\mathbb{R}^3$, whose $S_n$ representation is well-understood, see \cite{Pagaria} and the references within. 
\end{example}

\bibliography{matroid}
\bibliographystyle{amsalpha}

\end{document}